\documentclass[11pt,english]{amsart}
\usepackage[a4paper]{geometry}
\usepackage{amssymb}
\usepackage{amsmath}
\usepackage{fancyhdr}
\usepackage{color}
\usepackage{enumerate}
\newtheorem{theorem}{Theorem}[section] 
\newtheorem{lemma}[theorem]{Lemma}     
\newtheorem{corollary}[theorem]{Corollary}
\newtheorem{proposition}[theorem]{Proposition}

\newtheorem{example}[theorem]{Example}
\newtheorem{definition}[theorem]{Definition}
\newtheorem{remark}[theorem]{Remark}
\newtheorem{notation}[theorem]{Notation}
\newtheorem*{theorem*}{Theorem}

\usepackage[bookmarksopen, naturalnames]{hyperref}

\makeatother

\begin{document}

\title{Generalized universal series}
\author{S. Charpentier, A. Mouze, V. Munnier}
\address{St\'ephane Charpentier, 
Laboratoire d'Analyse, Topologie, Probabilit\'es, UMR 7353, Aix-Marseille Universit\'e, Technop\^ole Ch\^ateau-Gombert, CMI, 39 rue F. Joliot Curie,
13453 Marseille Cedex 13 FRANCE}
\email{stephane.charpentier1@univ-amu.fr}
\address{Augustin Mouze, Laboratoire Paul Painlev\'e, UMR 8524, Current address: \'Ecole Centrale de
Lille, Cit\'e Scientifique, CS20048, 59651 Villeneuve d'Ascq cedex}
\email{Augustin.Mouze@math.univ-lille1.fr}
\address{Vincent Munnier, \'Ecole Centrale de
Lille, Cit\'e Scientifique, CS20048, 59651 Villeneuve d'Ascq cedex}
\email{Vincent.Munnier@ec-lille.fr}

\begin{abstract} We unify the recently developed abstract theories of universal series and extended universal series 
to include sums of the form $\sum_{k=0}^n a_k x_{n,k}$ for given sequences of vectors $(x_{n,k})_{n\geq k\geq 0}$ 
in a topological vector space $X$. The algebraic and topological genericity as well as the spaceability are discussed. Then we provide various examples of such \emph{generalized} universal series which do not proceed from the classical theory. In particular, we build universal series involving Bernstein's polynomials, we obtain a universal series version of MacLane's Theorem, and we extend a result of Tsirivas concerning  universal Taylor series on simply connected domains, exploiting Bernstein-Walsh quantitative approximation theorem. 
\end{abstract}

\maketitle

\section{Introduction} In 1914, Fekete proved that there exists a formal Taylor series on $[-1,1]$ whose subsequences of partial sums approximate uniformly any function continuous on $[-1,1]$ which vanishes at $0$ \cite{Fekete-Pal}. Since then, several results of this 
type have appeared. In 1996, Nestoridis provided renewed vigor to the research on this kind of maximal divergence in \cite{Nes} by 
showing that there exists a power series $\sum_{k\geq 0}a_kz^k$ of radius of convergence $1$ such that for any compact set $K\subset \mathbb{C}\setminus \mathbb{D}$ (where $\mathbb{D}=\{z\in\mathbb{C};\ \vert z\vert<1\}$) with connected complement, and any function $h$ continuous on $K$ and holomorphic in the interior of $K$, there exists an increasing sequence $\left(\lambda _n\right)_n\subset \mathbb{N}$ such that
\[
\sup_{z\in K}\left\vert\sum_{k=0}^{\lambda _n}a_kz^k-h(z)\right\vert\rightarrow 0,\hbox{ as }n\rightarrow +\infty.
\]
These approximating series are now referred to as \emph{universal series}. A lot of other universal series appeared in different context (Faber series, Laurent series, Jacobi series...). We refer the reader to the excellent surveys \cite{GE} and \cite{Kah}.

\smallskip{}
The theory of universal series is a part of the theory of universality in operator theory which consists in the study of the orbit of a vector under the action of a sequence of operators. In fact universal series and hypercyclic operators constitute for now the most important instances of universality. The abstract theories of hypercyclicity and universal series respectively are now rather well-known (see \cite{BAY-MATH,bgnp} and the references therein) and provide a very efficient tool to exhibit new instances of hypercyclicity or universal series. But in a lot of situations we have to deal with objects which cannot be handled by these theories. For example we can be interested in studying the universal properties of the Ces\`aro means of the sequence of the iterates of a hypercyclic operator $T$ (\cite{LeonF}) or in studying the averaged sums 
$\left(\frac{1}{n}\sum_{k=0}^na_kz^k \right)_n .$ Yet these new objects cannot be considered from the point of view of hypercyclicity or universal series any more. 
Furthermore the previous abstract theory of universal series cannot provide a tool to study the action of summability processes on universal series. A very abstract theory of universality has been developed in 
\cite{GE1} and in Part C of \cite{bgnp} in order to provide criteria for universality and a description of the set of universal vectors. The point of view is very general and, so far, it was only used in the context of both universality and universal series in \cite{m}.

The purpose of this paper is to generalize the notion of classical universal series and to provide, like \cite{bgnp} for universal series, a rather complete abstract theory of \emph{generalized universal series} in order to significantly enlarge the class of universal objects which can be understood in a systematic way. Our theory is based on the more general under-exploited result of \cite[Part C]{bgnp}. We introduce the following notion of \emph{generalized universal series}. Given two convenient Fr\'echet spaces $E$ and $X$, a sequence $\left(e_k\right)_k\subset E$ and a double-indexed sequence $\left(x_{n,k}\right)_{n\geq k\geq 0}\subset X$ we say that $f\in E$ is a \emph{generalized universal series} if there exists $\left(a_k \right)_k\subset \mathbb{K}^{\mathbb{N}}$ ($\mathbb{K}=\mathbb{R}$ or $\mathbb{C}$) such that $f=\sum_k a_ke_k \in E$ and the set $\left\{\sum_{k=0}^na_kx_{n,k},\,n\geq k\geq 0\right\}$ is dense in $X$. This definition differs from that of a classical universal series by the only fact that a simple sequence $\left(x_{n}\right)_{n\geq 0}$ is replaced by a double-indexed sequence $\left(x_{n,k}\right)_{n\geq k\geq 0}$ but we will see that this modification is significant. When $E=\mathbb{K}^{\mathbb{N}}$ we will refer to a generalized universal series as a \emph{formal} generalized universal series.

\smallskip{}
Up to now, very specific results have been provided towards this direction. Hadjiloucas \cite{hadji} defined \emph{extended universal series} by requiring the set $\left\{\frac{1}{\phi(n)}\sum_{k=0}^na_kx_k,\,n\geq 0\right\}$ to be dense in $X$, where $\phi$ is an increasing function converging in $(0,+\infty]$. He extended the abstract results obtained for classical universal series in \cite{bgnp} but all his examples settled in $\mathbb{R}^{\mathbb{N}}$ or $\mathbb{C}^{\mathbb{N}}$ and directly came from classical universal series. In 2012, Tsirivas \cite{tsirivas} continued the study of \cite{hadji} by producing the first examples of extended universal Taylor series in simply connected domains. In particular, he linked the asymptotic behavior of $\phi$ to the existence of function $f$ holomorphic on a simply connected domain $\Omega$ such that any entire function is the uniform limit on every compact subset in $\Omega^c$, with connected complement, of a subsequence of $\left(\frac{1}{\phi(n)}\sum_{k=0}^na_k\left(z-\xi \right)^k\right)_n$, where $a_k=f^{(k)}(\xi)/k!$, $\xi \in \Omega$. Now, this essentially reduces to study generalized universal series where $x_{n,k}=\frac{x_k}{\phi(n)}$ where $\left(x_k \right)_k$ is a sequence in $X$. We also mention that Melas and Nestoridis exhibited large classes of power series which are universal and whose Ces\`aro means are also universal \cite{melanes1} (more general summability processes are considered). Notice that we unify all the previous results with our formalism.

\smallskip{}
In this paper we characterize the existence of generalized universal series in terms of easy-to-check approximating lemmas and, using \cite[Theorem 27]{bgnp} we obtain a description of the set of generalized universal series in terms of genericity ($G_{\delta}$-dense and algebraic genericity). We also adapt a general criterion given in \cite{CMM} for the spaceability of these sets, i.e. the existence of a closed infinite dimensional subspace of $E$ consisting of generalized universal series. When $E=\mathbb{K}^{\mathbb{N}}$ the existence of formal generalized universal series reduces to a very simple criterion and the spaceability can be characterized, extending \cite[Theorem 4.1]{CMM} (see below for more details). The results contained in \cite{hadji} - in which the spaceability was not discussed - then appear as a very particular case.
The rest of the paper is dedicated to examples of generalized universal series which are connected to classical phenomena of convergence or universality but which follow 
from neither hypercyclicity nor classical universal series.

The organization of the paper is as follows: Section \ref{S2} deals with the above mentioned abstract theory of generalized universal series. In Section \ref{ExGe}, we provide the first results about the genericity of the set of extended universal series in the sense of \cite{hadji}, in settings different from $\mathbb{R}^{\mathbb{N}}$ or $\mathbb{C}^{\mathbb{N}}$. A simple idea to produce potentially generalized universal series consists in considering partial sums of the form $\sum_{k=0}^n a_k\alpha _{n,k}x_k$, where $\left(\alpha _{n,k}\right)_{n\geq k\geq 0}$ is a sequence in $\mathbb{K}=\mathbb{R}\hbox{ or }\mathbb{C}$, and where the sequence $\left(x_k\right)_k$ induces classical universal series of the form $\sum_{k=0}^na_kx_k$. In particular, it is motivated by the natural question: how much or in which manner can we perturb a classical universal series in order to keep a universal object (to get a generalized universal series)? In Section \ref{ExGe}, we derive from Fekete's Theorem that, there exists a function $f\in \mathcal{C}^{\infty}$ such that every continuous functions on $\mathbb{R}$ which vanishes at $0$ can be uniformly approximated on 
each compact subset of $\mathbb{R}$ by some partial sums of the form 
$\sum_{k=0}^n\frac{f^{(k)}(0)}{k!}\alpha _{n,k}x^k,$ provided that the sequence $(\alpha_{n,k})_n$ is convergent for every $k\geq 0.$ 
We prove in an abstract fashion how the algebraic and topological genericity, as well as the spaceability, are preserved under such a perturbation. In the same section, we also build generalized universal series with respect to a sequence $\left(x_{n,k}\right)_{n\geq k\geq 0}$ of 
Bernstein's type polynomials (e.g. ${n\choose k}x^k(1-x)^{n-k}$, $n\geq k\geq 0$). We recall that 
this sequence is associated to any continuous function $g$ on $[0,1]$ and the sums 
$\sum_{k=0}^n{n\choose k}g(k/n)x^k(1-x)^{n-k}$ converge to $g(x)$ uniformly on $[0,1]$ 
as $n\rightarrow +\infty$ \cite{korovkin}. Thus, these Bernstein generalized universal series come in a natural way. 
Such a simple example is again specific to the formalism of generalized universal series and do not proceed from the classical theory. 

In every instance of classical universal series, the main idea is that whatever can be approximated by polynomials can also be approximated by some partial sums of a series. Then the notion of universal series cannot be but connected to approximation theory. Nevertheless, non-quantitative approximation theorems such as Weierstrass Theorem or Mergelyan Theorem were sufficient to construct classical universal series. In \cite{tsirivas}, the author used Bernstein's Walsh quantitative approximation theorem to build the extended universal series mentioned above. We improve his result by showing the existence of generalized universal series induced by a sequence $\left(x_{n,k}\right)_{n\geq k\geq 0}$ of the form $\left(\alpha_{n,k}z^k\right)_{n\geq k\geq 0}$. With such a kind of ``perturbations'', the situation is technically more involved. We obtain a result of the following type (see Section \ref{SectionBernstein}):
\begin{theorem*} Let $\left(\alpha_{n,k}\right)_{n\geq k\geq 0}$ be non-zero complex numbers. Denote by $\mathcal{U}_{\alpha}(\mathbb{D})$ (where $\mathbb{D}$ stands for the open unit disk) the set 
of holomorphic functions $f\in \mathcal{H}\left({\mathbb{D}}\right)$ such that, 
for every compact subset $K\subset \mathbb{C}$, $K\cap \mathbb{D}=\emptyset$, $K^c$ connected, every function $h$ continuous on $K$ and holomorphic in the interior of $K$, there exists a sequence $\left(\lambda _n\right)_n$ such that 
$\sum_{k=0}^{\lambda _n}\alpha_{n,k}\frac{f^{(k)}\left(0 \right)}{k!}z^k\rightarrow h(z),$ uniformly on $K,$ as $n\rightarrow +\infty.$ The following assertions hold:
\begin{enumerate}\item If the sequence $\left(\alpha_{n,k}\right)_{n}$ is convergent in $\mathbb{C}$ for every $k\geq 0$ and if we have
\begin{equation*}\limsup\left(\min _{0\leq k\leq n}\left\{\sqrt[n]{\left\vert \alpha _{n,k}\right\vert }\right\}\right)\geq 1,\end{equation*}
then $\mathcal{U}_{\alpha}(\mathbb{D})\ne\emptyset$;
\item If we have $\mathcal{U}_{\alpha}(\mathbb{D})\ne\emptyset$ then
\begin{equation*}\limsup\left(\max _{0\leq k\leq n}\left\{\sqrt[n]{\left\vert \alpha _{n,k}\right\vert }\right\}\right)\geq 1.\end{equation*}
\end{enumerate}
\end{theorem*}
In the case $\alpha_{n,k}=1/\phi(n)$ where $\phi(n)$ converges in $\mathbb{C}\setminus\{0\}\cup\{\infty\},$ 
the above result, in an equivalent form, was proven by Tsirivas in \cite{tsirivas}. Our result is more general and our proof is direct since we are using a precise version of Bernstein-Walsh's Theorem (see 
Theorem \ref{BWT}). The proof of this result shows that having approximating polynomials is no more sufficient to construct generalized universal series, but we also need a control on the degree of the approximating polynomials involved in the construction of the series. In addition, this control is also connected to the way in which the perturbation $\phi$ grows along some of its subsequences. This connection clearly appears in the proof of this theorem. This is the purpose of Section \ref{SectionBernstein}.

\smallskip{}
Some examples of generalized universal series also highlight other strange behaviors. In Section \ref{deriv-Section5}, we construct functions $f$ holomorphic around $0$ which are universal in the sense that the set $\left\{\alpha_n S_n\left(f^{(n)}\right),\,n\in \mathbb{N}\right\}$ is dense in $\mathcal{H}\left(\mathbb{C}\right)$, where $S_n(g)$ stands for the $n$-th partial sum of the Taylor expansion of $g$ at $0$ and $(\alpha_n)_n$ is a sequence of non zero complex numbers. When $\alpha _n=1$ for every $n$, it reminds MacLane's result concerning the hypercyclicity of the differentiation operator in $\mathcal{H}(\mathbb{C})$. In general, we relate in a precise way the radius of convergence of the Taylor expansion of $f$ at $0$ to the asymptotic behavior of the sequence $(\alpha_n)_n$.

\section{Abstract theory of generalized universal series}\label{S2}
\subsection{Background on universality}\label{S2S1}
The first part of this subsection is mostly inspired by Part C of \cite{bgnp} and Section 3 of \cite{CMM}. We will recall the formalism introduced in \cite{bgnp}.

Let $\mathbb{K}=\mathbb{R}$ or $\mathbb{C}$. Let $Y$ be a separable complete metrizable topological vector space (over $\mathbb{K}$) and $Z$ a metrizable topological vector space (over $\mathbb{K}$), whose topologies are induced by translation-invariant metrics $d_Y$ and $\varrho$, respectively. Let $M$ be a separable closed subspace of $Z$. Let $L_n:Y\rightarrow Z$, $n\in \mathbb{N}$, be continuous linear mappings.

\begin{definition}{\rm With the above notations, we say that $y\in Y$ is universal with respect to $\left(L_n\right)_n$ and $M$ if
\[
M\subset \overline{\left\{L_ny:\,n\in \mathbb{N}\right\}}.
\]
We denote by $\mathcal{U}\left(L_n,M\right)$ the set of such universal elements.}
\end{definition}

We make the following assumption.
\\
\textbf{Assumption (I).} There exists a dense subset $Y_0$ of $Y$ such that, for any $y\in Y_0$, $\left(L_ny\right)_n$ converges to an element in $M$.

We recall the following result from \cite{bgnp}.
\begin{theorem}\label{thms-27-and-28} (1) \emph{(}\cite[Theorems 26 and 27]{bgnp}\emph{)} Under Assumption (I), the following are equivalent.
\begin{enumerate}[(i)]
\item $\mathcal{U}\left(L_n,M\right)\neq \emptyset$;
\item For any open subset $U\neq \emptyset$ of $Y$ and any open subset $V\neq \emptyset$ of $Z$ with $V\cap M\neq \emptyset$ there is some $n\in \mathbb{N}$ with $L_n(U)\cap V\neq \emptyset$;
\item For every $z\in M$ and $\varepsilon>0$, there exist $n\geq 0$ and $y\in Y$ such that
\[
\varrho\left(L_ny,z\right)<\varepsilon\mbox{ and }d_Y\left(y,0\right)<\varepsilon;
\]
\item $\mathcal{U}\left(L_n,M\right)$ is a dense $G_{\delta}$ subset of $Y$.
\end{enumerate}
(2) \emph{(}\cite[Theorem 28 (1)]{bgnp}\emph{)} If, for every increasing sequence $\left(\mu_n\right)_n\subset \mathbb{N}$, $\mathcal{U}\left(L_{\mu_n},M\right)$ is non-empty, then $\mathcal{U}\left(L_n,M\right)$ contains, apart from $0$, a dense subspace of $Y$.
\end{theorem}

\begin{remark} {\rm When the sequence $\left(L_n\right)_n$ satisfies Condition (ii) in the above theorem, we say that $\left(L_n\right)_n$ is topologically transitive for $M$.}
\end{remark}

\noindent{}In the classical theory of universal series, it is useful to  assume that $\left(L_n\right)_n$ satisfies the following additional assumptions. 
\\
\textbf{Assumption (II).} There exist continuous linear maps $U_n:Y\rightarrow Y$, $n\in \mathbb{N}$, and a continuous map $P:Z\rightarrow Z$ with $P_{|M}=\mbox{id}_M$ such that
\begin{itemize}
\item For any $y\in Y$, $n\in \mathbb{N}$, $L_mU_ny\rightarrow PL_ny$ as $m\rightarrow +\infty$;
\item For any $y\in Y$, if $\left(L_{n_k}\right)_k$ converges to an element in $M$, then $U_{n_k}\rightarrow y$ as $k\rightarrow +\infty$.
\end{itemize}

We use the notations $L_{-1}=U_{-1}=0$. Under these extra assumptions, Theorem \ref{thms-27-and-28} can be improved as follows. We refer to \cite[Theorem 29]{bgnp} and \cite[Remark 27 (a)]{bgnp}.
\begin{theorem}\label{GEthm}Under Assumptions (I) and (II), the following assertions are equivalent.
\begin{enumerate}[(1)]
\item $\mathcal{U}\left(L_n,M\right)\neq \emptyset$;
\item $\mathcal{U}\left(L_n,M\right)$ is a dense $G_\delta$ subset of $Y$;
\item For every $z\in M$ and $\varepsilon>0$, there are $n\geq 0$ and $y\in Y$ such that
\[
\varrho\left(L_ny,z\right)<\varepsilon\quad \mbox{and}\quad d_Y\left(U_ny,0\right)<\varepsilon;
\]
\item For every $z\in M$, $\varepsilon>0$, and $p\geq -1$, there are $n>p$ and $y\in Y$ such that
\[
\varrho\left(L_ny-L_py,z\right)<\varepsilon\quad \mbox{and}\quad d_Y\left(U_ny-U_py,0\right)<\varepsilon;
\]
\item For every increasing sequence $\left(\mu_n\right)_n\subset \mathbb{N}$, $\mathcal{U}\left(L_{\mu_n},M\right)$ is a dense $G_\delta$ subset of $Y$;
\item For every increasing sequence $\left(\mu_n\right)_n\subset \mathbb{N}$, $\mathcal{U}\left(L_{\mu_n},M\right)$ contains, apart from $0$, a dense subspace of $Y$.
\end{enumerate}
\end{theorem}

The question of the spaceability of the set of universal elements has always been considered separately. Actually, it is more involved in general because the condition $\mathcal{U}\left(L_n,M\right)\neq \emptyset$ does not always imply that $\mathcal{U}\left(L_n,M\right)$ is spaceable. As far as we know, the most general criterion for the spaceability of the universal set is Proposition 3.7 in \cite{CMM}, given when $M=Z$. This result naturally extends to any subspace $M\subset Z$. To state this latter, we introduce a condition.

\begin{definition} {\rm With the above notations, assume that $Y$ is a Fr\'echet space. The sequence $\left(L_n\right)_n$ satisfies \emph{Condition (C)} for $M$ if there exist an increasing sequence $\left(n_k\right)_k\subset \mathbb{N}$ and a dense subset $Y_1\subset Y$ such that
\begin{itemize}\item For all $y\in Y_1$, $L_{n_k}y\rightarrow 0,$ as $k\rightarrow +\infty$;
\item For every continuous semi-norm $p$ on $Y$, $M\subset \overline{\cup_{k\geq 0}T_{n_k}\left(\left\{y\in Y:\,p(y)<1\right\}\right)}$.
\end{itemize}}
\end{definition}
Condition (C) has been introduced in \cite{Leon} for sequence of operators between Banach spaces, and extended in \cite{Men2} in the case of Fr\'echet spaces, when $M=Z$. In practice, Condition (C) is not quite easy to handle. In fact, it is closely related to a more common notion, that of mixing sequence of operators for a subspace.

\begin{definition} {\rm With the above notations, we say that the sequence $\left(L_n\right)_n$ is mixing for $M$ if the following condition is satisfied: for any open subset $U\neq \emptyset$ of $Y$ and any open subset $V\neq \emptyset$ of $Z$ with $V\cap M\neq \emptyset$ there is some $n\in\mathbb{N}$ such that $L_n(U)\cap V\neq \emptyset$ for any $n\geq N$.}
\end{definition}

The following immediate proposition (see for example \cite[Remark 27 (b)]{bgnp}) reformulates the notion of mixing sequence of operators for $M$.
\begin{proposition}\label{munmixing} With the above notations, if for every increasing sequence $\left(\mu_n\right)_n\subset \mathbb{N}$, $\mathcal{U}\left(L_{\mu_n},M\right)$ is non-empty, then the sequence $\left(L_n\right)_n$ is mixing for $M$.
\end{proposition}

A straightforward adaptation of \cite[Proposition 3.5]{CMM} yields the following chains of implications:

\begin{equation*}\hbox{ Mixing for }M\hbox{ }\Longrightarrow\hbox{ }\hbox{ Condition (C) for }M\hbox{ }\Longrightarrow\hbox{ }\hbox{ Topologically transitive for }M.
\end{equation*}

Now, \cite[Theorem 1.11 and Remark 1.12]{Men2} extend to universality for a subspace, using Condition (C) for a subspace, as follows. The proof is, up to very slight changes, quite identical to that of Menet's result, and so it is omitted.

\begin{proposition}\label{closedgene} With the above notations, we assume that $Y$ is a Fr\'echet space with continuous norm. Let $\left(p_n\right)_n$ be a non-decreasing sequence of norms and $\left(q_n\right)_n$ a non-decreasing sequence of semi-norms defining the topologies of $Y$ and $Z$ respectively. It the sequence $\left(L_n\right)_n$ satisfies Condition (C) for $M$ and if there exists a non-increasing sequence of infinite dimensional closed subspaces $\left(M_j\right)_j$ of $Y$ such that for every continuous semi-norm $q$ on $Z$, there exists a positive number $C$, an integer $k\geq 1$ and a continuous norm $p$ on $Y$ such that we have, for any $j\geq k$ and any $x\in M_j$
\[
q\left(L_{n_j}(x)\right)\leq Cp(x),
\]
then $\mathcal{U}\left(L_n,M\right)$ is spaceable.
\end{proposition}

\subsection{Generalized universal series}\label{S2S2}
In this paper, we are interested in generalizing the classical notion of universal series studied in \cite{bgnp}. Let $E$, $A\subset \mathbb{K}^{\mathbb{N}}$ and $X$ be three Fr\'echet spaces whose topologies are defined by translation-invariant metrics $d_E$, $d_A$ and $\varrho$, respectively. Let $\left(e_k\right)_k$ denote the canonical sequence in $\mathbb{K}^{\mathbb{N}}$. For $p$ a polynomial in $\mathbb{K}^{\mathbb{N}}$, i.e. a finite linear combination of the $e_k$'s, we denote by $v(p)$ its valuation (i.e. $v(p)=\min(k; a_k\ne 0)$ where $p=\sum_{k\geq 0}a_ke_k$) and by $d(p)$ its degree (i.e. $d(p)=\max(k; a_k\ne 0)$ where $p=\sum_{k\geq 0}a_ke_k$). Let $\mathcal{F}=\left(f_k\right)_{k\geq 0}$ be sequence in $E$ and let $\mathcal{X}=\left(x_{n,k}\right)_{n\geq k\geq 0}$ be a sequence in $X$. We make the following assumptions.

\noindent{}(i) The set $G\subset \mathbb{K}^{\mathbb{N}}$ of all polynomials is included in $A$ and is dense in $A$;\\
(ii) The coordinate projections $A\rightarrow \mathbb{K}$, $a=\left(a_n\right)_n\longmapsto a_k$, $k\geq 0$, are continuous;\\
(iii) The set $G_E\subset E$ of all polynomials in $E$ (i.e. the finite linear combinations of the $f_k$'s) is dense in $E$;\\
(iv) There exists a linear continuous map $T_0:E\rightarrow A$ such that $T_0\left(f_k\right)=e_k$ for every $k\geq 0$;\\
(v) For every $k\geq 0$, the sequence $\left(x_{n,k}\right)_{n\geq k}$ is convergent in $X$.

\smallskip
We also denote by $S_n^{\mathcal{F}}:A\rightarrow E$ (resp. $S_n^{\mathcal{X}}:A\rightarrow X$) the map which takes $\left(a_n\right)_n$ to $\sum_{k=0}^na_kf_k$ (resp. $\sum_{k=0}^na_kx_{n,k}$).

\begin{remark} {\rm One can observe that (iv) implies that the family $\left(f_k\right)_k$ is linearly independent.}
\end{remark}

\begin{definition} {\rm We keep the above assumptions. Let $\mu=\left(\mu_n\right)_{n\geq 0}\subset \mathbb{N}$ be an increasing sequence. \begin{enumerate}[(1)]
\item An element $f\in E$ is a $\mu$-generalized universal series (with respect to $\left(x_{n,k}\right)$) if
\[
X=\overline{\left\{S_{\mu_n}^{\mathcal{X}}\circ T_0(f):\,n\geq 0\right\}}.
\]
We denote by $\mathcal{U}^{\mu}\left(\mathcal{X}\right)\cap E$ the set of such elements.\\
\item An element $f\in E$ is a $\mu$-restricted generalized universal series (with respect to $\left(f_k\right)$, $\left(x_{n,k}\right)$) if
\[
X\times\{0\}\subset \overline{ \left\{\left(S_{\mu_n}^{\mathcal{X}}\circ T_0(f),f-S_{\mu_n}^{\mathcal{F}}
\circ T_0(f)\right):\,n\geq 0\right\} }.
\]
We denote by $\mathcal{U}^{\mu}_E\left(\mathcal{X}\right)$ the set of such elements.\end{enumerate}}
\end{definition}

\begin{remark} {\rm (1) Observe that the indices $n$ and $k$ in $\left(x_{n,k}\right)_{n\geq k\geq 0}$ do not play the same role in the definition of generalized universal series.\\
(2) When $x_{n,k}=x_k$ for every $n\geq 0$ and every $0\leq k\leq n$, we recover the standard notion of universal (resp. restricted universal) series with respect to $\left(x_k\right)_k$.\\
(3) When $\left(f_k\right)_k$ is a Schauder basis of $E$, we automatically have $\mathcal{U}^{\mu}\cap E\left(\mathcal{X}\right)=\mathcal{U}^{\mu}_E\left(\mathcal{X}\right)$.}
\end{remark}

\begin{notation} {\rm When $\mu=\mathbb{N}$, we simply denote $\mathcal{U}^{\mu}\left(\mathcal{X}\right)\cap E$ (resp. $\mathcal{U}^{\mu}_E\left(\mathcal{X}\right)$) by $\mathcal{U}\left(\mathcal{X}\right)\cap E$ (resp. $\mathcal{U}_E\left(\mathcal{X}\right)$). When  $x_{n,k}=x_k$ for every $n\geq 0$ and every $0\leq k\leq n$, we use the standard notations $\mathcal{U}^{\mu}\cap E$ and $\mathcal{U}^{\mu}_E$ to denote the set of classical universal series and that of classical restricted universal series respectively \cite{bgnp}.}
\end{notation}

\noindent The next result characterizes the existence of formal generalized universal series (i.e. when $E=A=\mathbb{K}^{\mathbb{N}}$).

\begin{proposition}\label{prop1} Let $\mu =\left(\mu _n\right)_n\subset \mathbb{N}$ be an increasing sequence. Under the previous notations and assumptions, the following assertions are equivalent.
\begin{enumerate}[(1)]
\item $\mathcal{U}^{\mu}\left(\mathcal{X}\right)\cap \mathbb{K}^{\mathbb{N}} \neq \emptyset$;
\item For all $K\geq 0$, $\displaystyle{\bigcup_{n\in \mu}\emph{span}\left(x_{n,k},\,K\leq k\leq n\right)}$ is dense in $X$;
\item For all $N\geq K\geq 0$, $\displaystyle{\bigcup_{\substack{n\geq N \\ n\in \mu}}\emph{span}\left(x_{n,k},\,K\leq k\leq n\right)}$ is dense in $X$.
\end{enumerate}
\end{proposition}

\begin{proof}{\it (2)$\Longleftrightarrow$(3):} The implication (3)$\Rightarrow$(2) is obvious. Then, we assume that (2) is satisfied. Let $\left(K,N\right)\in \mathbb{N}^2$, $K\leq N$, be fixed. By (2),
\begin{eqnarray*}X & = & \overline{\bigcup_{n\in \mu}\text{span}\left(x_{n,k},\,N\leq k\leq n\right)}\\
& = & \overline{\bigcup_{\substack{n\geq N \\ n\in \mu}}\text{span}\left(x_{n,k},\,N\leq k\leq n\right)}\\
& \subset & \overline{\bigcup_{\substack{n\geq N \\ n\in \mu}}\text{span}\left(x_{n,k},\,K\leq k\leq n\right)},
\end{eqnarray*}
which gives the conclusion.

{\it (1)$\Rightarrow$(2):} Let $\left(U_j\right)_{j\geq 0}$ be a denumerable basis of open sets in $X$ and let us fix $j_0$. Let also $K\geq 0$ be fixed. Let $a=\left(a_n\right)_n\in \mathcal{U}^{\mu}\left(\mathcal{X}\right)$. Let $p=\left(a_0,a_1,\ldots,a_K,0,\ldots\right)$ and $p_{n}=\sum_{k=0}^Ka_kx_{n,k}$. By assumption (v) above, $p_n$ has a limit in $X$ when $n$ tends to $+\infty$. Hence $a-p=\left(0,\ldots,0,a_K,\ldots,a_n,\ldots\right)\in \mathcal{U}^{\mu}\left(\mathcal{X}\right)$ also. Therefore, there exists $n\geq K$, $n\in \mu$, such that $\sum_{k=K}^na_kx_{n,k}\in U_{j_0}$, which gives (2).

{\it (2)$\Rightarrow$(1):} The assumption allows to build a sequence of polynomials $\left(p_j\right)_{j\geq 0}$ with $d\left(p_j\right)\in \mu$ for any $j\geq 0$, $v\left(p_j\right)>d\left(p_{j-1}\right)$, and such that $S_{d\left(p_{j}\right)}^{\mathcal{X}}\left(p_{j}\right)\in U_{j}-\sum_{i=0}^{j-1}S_{d\left(p_{i}\right)}^{\mathcal{X}}\left(p_i\right)$. Then the (block) sequence $\left(p_j\right)_j\in \mathbb{K}^{\mathbb{N}}$ belongs to $\mathcal{U}^{\mu}\left(\mathcal{X}\right)\cap \mathbb{K}^{\mathbb{N}}$.
\end{proof}

\begin{remark} {\rm When  $x_{n,k}=x_k$ for every $n\geq 0$ and every $0\leq k\leq n$, we recover that $\mathcal{U}\cap \mathbb{K}^{\mathbb{N}}$ is non-empty if and only if $\hbox{span}\left(x_k,\,k\geq n\right)$ is dense in $X$ for every $n\geq 0$.}
\end{remark}

We are interested in a description of the set of generalized universal series. To this purpose, the results of Subsection \ref{S2S1} are very useful. With the notations of this latter, we will consider the two following settings, depending on what we are considering: generalized universal series or restricted generalized universal series. We fix an increasing sequence $\mu=\left(\mu_n\right)_n\subset \mathbb{N}$.
\begin{enumerate}[(A)]\item For generalized universal series:
\begin{enumerate}[{A}-(a)]
\item $Y=E$, $Z=M=X$;
\item $L_n=\left(S_{\mu_n}^{\mathcal{X}}\circ T_0\right)$, $n\in\mathbb{N}$;
\item $Y_0=G_E$, where we recall that $G_E$ is the set of all finite combinations of the $f_k$'s.
\end{enumerate}
\item For restricted generalized universal series:
\begin{enumerate}[{B}-(a)]
\item $Y=E$, $Z=X\times E$ and $M=X\times \{0\}$
\item $L_n=\left(S_{\mu_n}^{\mathcal{X}}\circ T_0,\mbox{id}_E-S_{\mu_n}^{\mathcal{F}}\circ T_0\right)$, $n\in\mathbb{N}$;
\item $Y_0=G_E$.
\end{enumerate}
\end{enumerate}
It is easily checked that Assumption (I) (see above) is satisfied under (A) or (B). The main result of this section states as follows.

\begin{theorem}\label{thm-gen-approx-lem}Let $\mu=\left(\mu_n\right)_n\subset \mathbb{N}$ be an increasing sequence.
\begin{enumerate}[(A)]\item For generalized universal series, we have:
\begin{enumerate}[{(A)}-(1)]\item The following assertions are equivalent.\\
(i) $\mathcal{U}^{\mu}\left(\mathcal{X}\right)\cap E\neq \emptyset$;\\
(ii) For every $x\in X$ and every $\varepsilon>0$, there exist $n\geq 0$ and $f\in E$ such that
\[
\varrho\left(S_{\mu _n}^{\mathcal{X}}\circ T_0(f),x\right)<\varepsilon\mbox{ and }d_E\left(f,0\right)<\varepsilon;
\]\\
(iii) $\mathcal{U}^{\mu}\left(\mathcal{X}\right)\cap E$ is a dense $G_\delta$ subset of $E$.
\item If, for every increasing sequence $\nu =\left(\nu _n\right)_n\subset \mu$, $\mathcal{U}^{\nu}\left(\mathcal{X}\right)\cap E$ is non-empty, then $\mathcal{U}^{\mu}\left(\mathcal{X}\right)\cap E$ contains, apart from $0$, a dense subspace of $E$.
\item If $E$ admits a continuous norm and if, for every increasing sequence $\nu =\left(\nu _n\right)_n\subset \mu$, $\mathcal{U}^{\nu}\left(\mathcal{X}\right)\cap E$ is non-empty, then $\mathcal{U}^{\mu}\left(\mathcal{X}\right)\cap E$ is spaceable.
\end{enumerate}
\item For restricted generalized universal series, we have:
\begin{enumerate}[{(B)}-(1)]\item The following assertions are equivalent.\\
(i) $\mathcal{U}^{\mu}_E\left(\mathcal{X}\right)\neq \emptyset$;\\
(ii) For every $x\in X$ and every $\varepsilon>0$, there exist $n\geq 0$ and $f\in E$ such that
\[
\varrho\left(S_{\mu _n}^\mathcal{X}\circ T_0(f),x\right)<\varepsilon\mbox{ and }d_E\left(S_{\mu _n}^{\mathcal{F}}\circ T_0(f),0\right)<\varepsilon;
\]
(iii) $\mathcal{U}^{\mu}_E\left(\mathcal{X}\right)$ is a dense $G_\delta$ subset of $E$.
\item  If, for every increasing sequence $\nu =\left(\nu _n\right)_n\subset \mu$, $\mathcal{U}^{\nu}_E\left(\mathcal{X}\right)$ is non-empty, then $\mathcal{U}^{\mu}_E\left(\mathcal{X}\right)$ contains, apart from $0$, a dense subspace of $E$.
\item If $E$ admits a continuous norm and if, for every increasing sequence $\nu =\left(\nu _n\right)_n\subset \mu$, $\mathcal{U}^{\nu}_E\left(\mathcal{X}\right)$ is non-empty, then $\mathcal{U}^{\mu}_E\left(\mathcal{X}\right)$ is spaceable.
\end{enumerate}\end{enumerate}
\end{theorem}

\begin{proof}Parts (A)-(1), (A)-(2), (B)-(1) and (B)-(2) are direct applications of Theorem \ref{thms-27-and-28}. We just mention that, to get (B)-(1)-(ii), we use the density of polynomials (with respect to the $f_k$'s) in $E$ and the continuity of the maps $S_{\mu_n}^{\mathcal{F}}\circ T_0$, $n\geq 0$.\\
It remains to prove (A)-(3) and (B)-(3). Both of these assertions will proceed from Proposition \ref{closedgene}. By Proposition \ref{munmixing}, $\left(L_n\right)_n$ is mixing for $M$ in both cases (in case (A), $M=X$, in case (B), $M=X\times\{0\}$). In particular, it satisfies Condition (C) for some increasing sequence $\left(n_k\right)_k\subset \mathbb{N}$. For $j\geq 0$, let $M_j=\cap_{k=0}^{n_j}\ker S_k^{\mathcal{X}}\circ T_0$. Every $M_j$ is an infinite dimensional closed subspace of $E$, and (A)-(3) immediately comes from an application of Proposition \ref{closedgene}. For (B)-(3), we just need to observe that for any seminorm $q$ of the Fr\'echet space $X\times E$, there actually exists a continuous norm $p$ of $E$ such that for any $f\in M_j$,
\[
q\left(L_{n_j}(f)\right)=q\left(0,f-S_{n_j}^{\mathcal{F}}(f)\right)\leq Cp(f),
\]
for any $j\geq 0$, by continuity of $S_{n_j}^{\mathcal{F}}$.
\end{proof}

\begin{remark}{\rm (1) When the sequence $\left(f_k\right)_k$ is a Schauder basis of $E$, we shall say that Part A and Part B of the previous theorem coincide.\\
(2) Using that the set of all finite combinations of the $f_k$'s is dense in $E$, it easily stems that Assertion (A)-(1) (ii) can be rephrased as follows, using the fact that the generalized universal property is preserved under translation by polynomials:

(A)-(1) (ii)' For every $x\in X$ and every $\varepsilon>0$, there exist $m\geq n\geq 0$ such that
\[
\varrho\left(S_{\mu _n}^{\mathcal{X}}\circ T_0(f),x\right)<\varepsilon\mbox{ and }d_E\left(S_{\mu _m}^{\mathcal{F}}\circ T_0(f),0\right)<\varepsilon.
\]}
\end{remark}

On the one hand, if  $x_{n,k}=x_k$ for every $n\geq 0$ and every $0\leq k\leq n$, then we observe that Assumption (II) in Subsection \ref{S2S1} is satisfied in the setting of restricted universal series, that is under (B). Therefore, Theorem \ref{GEthm} immediately yields \cite[Theorem 1]{bgnp} which asserts, in particular, that if $\mathcal{U}_E$ is non-empty, then it automatically contains, apart from $0$, a dense subspace. In this context, \cite{Men} also ensures that if $E$ admits a continuous norm and if $\mathcal{U}_E$ is non-empty, then $\mathcal{U}_E$ is spaceable. On the other hand, the formalism of generalized universal series is too much general to proceed from Theorem \ref{GEthm} and \cite{Men} (or \cite{charp}).

\subsection{Spaceability of generalized universal series in $\mathbb{K}^{\mathbb{N}}$} Let us return in this section to the space $\mathbb{K}^{\mathbb{N}}=\mathbb{R}^{\mathbb{N}}$ or $\mathbb{C}^{\mathbb{N}}$ endowed with the Cartesian topology. It is well known that 
$\mathbb{K}^{\mathbb{N}}$ does not admit a continuous norm, so the spaceability of the set of generalized universal series in $\mathbb{K}^{\mathbb{N}}$ cannot proceed from Theorem \ref{thm-gen-approx-lem}. We keep the previous notations: let $X$ be a metrizable vector space over the field 
$\mathbb{K}=\mathbb{R}$ or $\mathbb{C}.$ 
Let us denote $\mathcal{X}=\left(x_{n,k}\right)_{n\geq k\geq 0}$ a fixed sequence of elements in $X.$ In this context 
we always have $\mathcal{U}_{\mathbb{K}^{\mathbb{N}}}(\mathcal{X})=\mathcal{U}(\mathcal{X})\cap \mathbb{K}^{\mathbb{N}}.$ We are interested in the spaceability of the set $\mathcal{U}(\mathcal{X})\cap {\mathbb{K}^{\mathbb{N}}}$. In the case of classical universal series, a complete answer is given by Theorem 4.1 of \cite{CMM}. 
The general case seems much more complicated. In the next proposition, we give sufficient conditions for 
the spaceability (resp. the non spaceability) of $\mathcal{U}(\mathcal{X})\cap {\mathbb{K}^{\mathbb{N}}}$, 
which will cover the new examples of generalized universal sets. First, 
for every $l\geq 1,$ let us introduce the sets
\begin{multline*}E_l=\{b\in X;\,\exists (n,m)\in\mathbb{N}^2\hbox{ with }n>m\geq l\hbox{ and }
b_l,b_{l+1},\dots,b_m,\dots,b_n\in\mathbb{K}\hbox{ such that }\\
\displaystyle b=\sum_{j=l}^mb_jx_{m,j}\hbox{ and }
\sum_{j=l}^mb_jx_{n,j}=\sum_{j=m+1}^nb_jx_{n,j}\}
\end{multline*}
and
\begin{multline*}T_l=\{b\in X;\,\exists (n,m)\in\mathbb{N}^2\hbox{ with }n>m\geq l
\hbox{ and }
b_l,b_{l+1},\dots,b_m,\dots,b_n,b_{n+1},\dots\in\mathbb{K}\hbox{ such that }\\
\displaystyle b=\sum_{j=l}^mb_jx_{m,j}\hbox{ and for every } k\geq n,\ 
\sum_{j=l}^mb_jx_{k,j}=\sum_{j=m+1}^kb_jx_{k,j}\}\end{multline*}
We are ready to state the following propositions.

\begin{proposition}\label{propospacesequence1} If for every $l\geq 1$ the set $T_l$ is dense in $X,$ then 
$\mathcal{U}(\mathcal{X})\cap \mathbb{K}^{\mathbb{N}}$ is spaceable.

\end{proposition}
\begin{proof} Assume that, for every $l\geq 1$,  
the set $T_l$ is dense in $X$. In particular, $X$ is separable and we can consider a countable basis of open sets 
$(U_p)$ in $X$. We want to construct a sequence $(u_p)_{p\geq 0}$ in $\mathbb{K}^{\mathbb{N}}$ 
such that the sequence of valuations $(v(u_p))_{p\geq 0}$ is strictly increasing and 
all the non-zero elements $\sum_{p\geq 0}\alpha_p u_p$ are universal. Using the hypothesis, 
for any $p\geq 0$, any $l\geq 0,$ there exist $b\in\mathbb{K}^{\mathbb{N}},$ $m\geq l,$ $n\geq m$ and 
$b_l,\dots,b_n,b_{n+1},\dots\in\mathbb{K}$ such that 
\[
l\leq v(b),\ S_m^{\mathcal{X}}(b)\in U_p \hbox{ and } \forall k\geq n\ 
\sum_{j=l}^mb_jx_{k,j}=\sum_{j=m+1}^kb_jx_{k,j}.
\]
Define the sequence $a=(0,\dots,0,b_l,\dots,b_{m},-b_{m+1},\dots,-b_{n},-b_{n+1},\dots).$ 
Therefore we have 
\[
l\leq v(a),\ S^{\mathcal{X}}_m(a)\in U_k \hbox{ and } \forall k\geq n\  S^{\mathcal{X}}_{k}(a)=0.
\] 
We finish the proof along the same lines as that of $(2)\Rightarrow (1)$ of \cite[Theorem 4.1]{CMM}.
\end{proof}

\begin{proposition}\label{propospacesequence2} Assume that $X$ admits a continuous norm $\left\Vert \cdot \right\Vert.$ 
Then, if there exists $l\geq 1$ such that $E_l$ is not dense in 
$X,$ then $\mathcal{U}(\mathcal{X})\cap \mathbb{K}^{\mathbb{N}}$ is not spaceable.
\end{proposition}

\begin{proof} First assume that there exists $l\geq 0$ such that 
the set $E_l$ 
is not dense in $X.$ Therefore there exists an open set $U\subset X$ such that $U\cap E_l=\emptyset.$ 
Suppose that there exists an infinite dimensional closed subspace $F$ in 
$\mathcal{U}(X)\cap \mathbb{K}^{\mathbb{N}}.$ One can find a sequence $(u^{(n)})_{n\geq 0}$ in $F\setminus\{0\}$ 
such that $(v(u^{(n)}))_{n\geq 0}$ is strictly increasing and $v(u^{(0)})\geq l$ \cite[Lemma 5.1]{charp}. Since for every 
$n\geq 0$, $u^{(n)}$ is a universal element, there exists $m_n\geq v\left(u^{(n)}\right)$ such that 
$S_{m_n}^{\mathcal{X}}(u^{(n)})=\sum_{j=l}^{m_n}u^{(n)}_jx_{m_n,j}\in U.$ If there exists 
$k>m_n$ such that $S_{k}(u^{(n)})\ne 0,$ then we have 
$\sum_{j=l}^{m_n}u^{(n)}_jx_{k,j}=-\sum_{j=m_n+1}^{k}u^{(n)}_jx_{m_k,j},$ i.e. $u^{(n)}\in E_l,$ 
which is impossible. 
Thus we have $S_k^{\mathcal{X}}(u^{(n)})\ne 0$ for all $k>m_n.$ 
We finish the proof along the same lines as that of $(1)\Rightarrow (2)$ of \cite[Theorem 4.1]{CMM}.
\end{proof}

\begin{example}\label{applicationspropospacesequence}{\rm (1) When $x_{n,k}=(1/\varphi(n))x_k$ for every $n\geq 0$, where $\varphi$ is an increasing function $\mathbb{N}\rightarrow\mathbb{K}\setminus \{0\}$ converging in $\mathbb{K}\setminus \{0\}\cup \{\infty\},$ i.e. in the case of extended universal series \cite{hadji} (see Example \ref{1st-exs-gen} (1)), observe that the set $T_l$ is dense in $X$ if and only of 
the set $E_l$ is dense in $X.$ In particular, we obtain a characterization of the spaceability of
sets of extended universal series in $\mathbb{K}^{\mathbb{N}}$. 
For example, Theorem \ref{thm-gen-approx-lem} ensures that there exists a universal sequence $a=(a_n)_{n\geq 0}$ 
of real numbers such that the set of all partial sums $\frac{1}{n}\sum_{j=0}^na_j$ is dense in $\mathbb{R}$ 
and that the set of such sequences is spaceable. Moreover it is easy to check that all the examples of sets of 
extended universal series given in \cite{hadji} are not spaceable.

(2) Further using the fact that there exists a universal sequence $a=(a_n)_{n\geq 0}$ 
of real numbers such that the set of all partial sums $\frac{1}{p}\sum_{j=0}^{p-1}a_j,$ when $p$ is a 
prime number, is dense in $\mathbb{R},$ observe that we can construct a Lebesgue-measurable function $f$ such that its 
Riemann sums $\left(\frac{1}{n}\sum_{j=0}^{n-1}f\left(\frac{j}{n}\right)\right)_n$ are dense in $\mathbb{R}.$ To do this, it suffices 
to set $f\left(\frac{j}{p}\right)=a_j,$ for $j=0,\dots, p-1,$ if $p$ is a prime number, and $f(t)=0$ otherwise. Observe that $f=0$ 
Lebesgue almost-everywhere. 

(3) Theorem \ref{thm-gen-approx-lem} ensures that there exists a universal sequence $a=(a_n)_{n\geq 0}$ 
of real numbers such that the set of all partial sums $\sum_{j=0}^n\frac{a_j}{n+j}$ is dense in $\mathbb{R}$ 
and Proposition \ref{propospacesequence1} ensures that the set of such sequences is spaceable. 
This example is not covered by the extended abstract theory. Further notice that  
it is possible to argue as in the above assertion (2) to find a Lebesgue-integrable function $t\mapsto \frac{f(t)}{1+t}$ such 
that the sequence of all its 
Riemann sums $\left(\sum_{j=0}^{n-1}\frac{f(\frac{j}{n})}{j+n}\right)_n$ is dense in $\mathbb{R}.$

(4) If for every integer $n$ the family $(x_{n,k})_{n\geq k\geq 0}$ is a free 
family, then Proposition 
\ref{propospacesequence2} ensures that the set 
$\mathcal{U}(\mathcal{X})\cap \mathbb{K}^{\mathbb{N}}$ is not spaceable. 

}
\end{example}

\subsection{A large class of generalized universal series} The easier way to produce 
generalized universal series is the following. Let $\mathcal{X}=\left(x_k\right)_{k\geq 0}$ be a sequence in $X$ and let $\alpha=\left(\alpha _{n,k}\right)_{n\geq k\geq 0}$ and $\beta=\left(\beta _{n,k}\right)_{n\geq k\geq 0}$ be two sequences of non-zero elements of $\mathbb{K}$, such that for any $k\geq 0$, the sequences $\left(\alpha _{n,k}\right)_{n}$ and $\left(\beta _{n,k}\right)_{n}$ are convergent in $\mathbb{K}$. Roughly speaking, the next proposition asserts that, whenever there exists a formal generalized universal series with respect to $\alpha\mathcal{X}=\left(x_{n,k}\right)_{n\geq k\geq 0}:=\left(\alpha_{n,k}x_k\right)_{n\geq k\geq 0}$, there also exists a formal generalized universal series with respect to $\beta\mathcal{X}=\left(x_{n,k}\right)_{n\geq k\geq 0}:=\left(\beta _{n,k}x_k\right)_{n\geq k\geq 0}$.

\begin{proposition}\label{prop2}Let $\mu \subset\mathbb{N}$ and $\nu \subset\mathbb{N}$ be two increasing sequences of natural numbers. The set $\mathcal{U}^{\mu}\left(\alpha\mathcal{X}\right)\cap \mathbb{K}^{\mathbb{N}}$ is non-empty if and only if the set $\mathcal{U}^{\nu}\left(\beta\mathcal{X}\right)\cap \mathbb{K}^{\mathbb{N}}$ is non-empty.
\end{proposition}
\begin{proof}Without loss of generality, it suffices to prove that $\mathcal{U}^{\mu}\left(\alpha\mathcal{X}\right)\cap \mathbb{K}^{\mathbb{N}}\neq \emptyset$ implies that $\mathcal{U}^{\nu}\left(\beta\mathcal{X}\right)\cap \mathbb{K}^{\mathbb{N}}\neq \emptyset$. Let $\left(U_j\right)_{j\geq 0}$ be a denumerable basis of open sets in $X$ and let us fix $j_0$. Let also $K\geq 0$. According to Proposition \ref{prop1}, there exists $n\geq K$, $n\in \mu$, such that $\sum_{k=K}^na_k\alpha_{n,k}x_{k}\in U_{j_0}$. Let $m$ be an integer in $\nu$ such that $m\geq n$. Then we have $\sum_{k=K}^ma_k\frac{\alpha_{n,k}}{\beta _{m,k}}\beta _{m,k}x_{k}\in U_{j_0}$, with $a_k=0$ for every $n<k\leq m$. By Proposition \ref{prop1}, we deduce that $\mathcal{U}^{\nu}\left(\beta\mathcal{X}\right)\cap \mathbb{K}^{\mathbb{N}}\neq \emptyset$.
\end{proof}

Denoting by $\mathcal{U}^{\mu}$ the set of formal classical universal series (i.e. $\mathcal{U}^{\mu}=\mathcal{U}^{\mu}\left(\beta\mathcal{X}\right)$ with $\beta _{n,k}=1$ for every $n\geq k\geq 0$), we immediately deduce the following corollary, using Theorem \ref{thm-gen-approx-lem}.
\begin{corollary}\label{dense-subspace-K-N}Let $\mu \subset \mathbb{N}$ be an increasing sequence. The following assertions are equivalent.
\begin{enumerate}\item $\mathcal{U}^{\mu}\cap \mathbb{K}^{\mathbb{N}}$ is non-empty;
\item For every increasing sequence $\nu \subset \mathbb{N}$, $\mathcal{U}^{\nu}\left(\alpha\mathcal{X}\right)\cap \mathbb{K}^{\mathbb{N}}$ is non-empty;
\item For every increasing sequence $\nu \subset \mathbb{N}$, $\mathcal{U}^{\nu}\left(\alpha\mathcal{X}\right)\cap \mathbb{K}^{\mathbb{N}}$ is a dense $G_{\delta}$-subset of $\mathbb{K}^{\mathbb{N}}$ and contains, apart from $0$, a dense subspace.
\end{enumerate}
\end{corollary}

\medskip{}
\begin{remark}\rm{In view of Proposition \ref{prop2}, given two sequences $\alpha$ and $\beta$, we can wonder whether the following equality holds:
$$\mathcal{U}^{\mu}\left(\alpha\mathcal{X}\right)\cap \mathbb{K}^{\mathbb{N}}=\mathcal{U}^{\mu}\left(\beta\mathcal{X}\right)\cap \mathbb{K}^{\mathbb{N}}?$$
We show with an example that the answer to this question is negative. With the previous notations, we take $X=\mathbb{R}$ and we consider a sequence $\mathcal{X}=\left(x_{n,k}\right)_{n\geq k\geq 0}\subset \mathbb{R}\setminus \{0\}$. Let $q:\mathbb{N}\rightarrow \mathbb{Q}$ a one-to-one and onto sequence such that $q_0\neq 0$. We define a sequence $\left(a_k\right)_{k\geq 0}\subset \mathbb{R}$ as follows: $a_0=q_0$ and if we assume that $a_0,\ldots,a_n$ have been built, we define inductively $a_{n+1}$ so that
$$q_{n+1}=\sum _{k=0}^{n+1}a_k x_{n+1,k}.$$
Hence the numbers $a_k$, $k\geq 0$, are defined in order to have
$$q_n=\sum _{k=0}^na_kx_{n,k}\text{ for every }n=0,1,\ldots,$$
so that the sequence $\left(\sum _{k=0}^na_kx_{n,k}\right)_n$ is dense in $\mathbb{R}$, and then $\left(a_n\right)_n \in \mathcal{U}(\mathcal{X})\cap \mathbb{R}^{\mathbb{N}}$.

Notice that there are infinitely many non-zero $a_k$. We can define the sequence $\beta=\left(\beta_{n,k}\right)_{n\geq k\geq 0}$ in $\mathbb{R}$ as follows: For every $n\in \mathbb{N}$ and $0\leq k\leq n$, we set $\beta_{n,k}=0$ if $a_k=0$ and $\beta_{n,k}=1/a_k$ otherwise. Therefore $\sum _{k=0}^na_k\beta_{n,k}\in \mathbb{N}$ for every $n\in \mathbb{N}$ so, taking $x_{n,k}=1$ for every $0\leq k\leq n$, we get $\mathcal{U}(\mathcal{X})\cap \mathbb{R}^{\mathbb{N}}\neq \mathcal{U}(\beta\mathcal{X})\cap \mathbb{R}^{\mathbb{N}}$.}
\end{remark}

\medskip{}
In the case $E=A=\mathbb{K}^{\mathbb{N}}$ we also have the following. 

\begin{proposition} Let $\mu \subset\mathbb{N}$ be an increasing sequence and let $\mathcal{X}:=\left(x_{n,k}\right)_{n\geq k\geq 0}$ and $\mathcal{Y}:=\left(y_{k}\right)_{k\geq 0}$ be such that $x_{n,k}\rightarrow y_k$ as $n\rightarrow +\infty$ for every $k\geq 0$. If $\mathcal{U}^{\mu}(\mathcal{Y})\cap\mathbb{K}^{\mathbb{N}}\ne\emptyset$ then 
$\mathcal{U}^{\mu}(\mathcal{X})\cap\mathbb{K}^{\mathbb{N}}\ne\emptyset$ but the converse does not hold.
\end{proposition}

\begin{proof}The first assertion is obvious by property $(2)$ of Proposition \ref{prop1}, the fact that $x_{n,k}\rightarrow y_k$ as $n\rightarrow +\infty$ 
for every $k\geq 0$ and the triangle inequality. 

We now prove that $\mathcal{U}^{\mu}(\mathcal{X})\cap\mathbb{K}^{\mathbb{N}}\ne\emptyset$ does not necessarily imply $\mathcal{U}^{\mu}(\mathcal{Y})\cap\mathbb{K}^{\mathbb{N}}\ne\emptyset$. Observe that if $y_k=0$ for every $k\geq 0$ except a finite set of natural numbers then 
by property $(2)$ of Proposition \ref{prop1} we have $\mathcal{U}^{\mu}(\mathcal{Y})\cap\mathbb{K}^{\mathbb{N}}=\emptyset.$ To obtain the desired conclusion, it 
suffices to construct 
a sequence $\mathcal{X}=\left(x_{n,k}\right)_{n\geq k\geq 0}$ where we have $\mathcal{U}^{\mu}(\mathcal{X})\cap\mathbb{K}^{\mathbb{N}}\ne\emptyset.$ So, using the previous notations, we take $X=\mathbb{R},$ 
endowed with the usual topology, and $\mathbb{K}=\mathbb{R}.$ Let $q:\mathbb{N}\rightarrow \mathbb{Q}$ be 
a sequence that is one-t-one and onto. We inductively define the sequence $\lambda:=(\lambda_n)_n \subset \mathbb{R}$ as follows: 
$$\lambda_0=q_0\hbox{ and, for }n=1,2,\dots,\, q_n=\frac{\lambda_0 +\lambda_1+\dots +\lambda_n}{n}.$$
For every $n=1,2,\dots,$ we set 
$$x_{n,k}=\frac{1}{n},\ k=0,1,2,\dots,n.$$
Observe that we have 
$$\sum_{j=0}^n\lambda_jx_{n,j}=q_n,$$
which implies that the sequence $(\sum_{j=0}^n\lambda_jx_{n,j})_n$ is dense in $\mathbb{R}.$ 
It follows that $\mathcal{U}(\mathcal{X})\cap \mathbb{R}^{\mathbb{N}}\ne\emptyset$ and since $x_{n,k}\rightarrow 0,$ 
as $n\rightarrow +\infty,$ for every $k\in \mathbb{N}$, the proof is complete.

\end{proof}

\begin{example}\label{1st-exs-gen}{\rm (1) $\alpha_{n,k}=1/\varphi(n)$ for every $n\geq 0$, where $\varphi$ is an increasing function $\mathbb{N}\rightarrow\mathbb{K}\setminus \{0\}$ converging in $\mathbb{K}\setminus \{0\}\cup \{\infty\}$. As we already said, this class of examples has been studied in \cite{hadji}. His main result \cite[Theorem 3.1]{hadji} is a consequence of the more general Theorem \ref{thm-gen-approx-lem} and Corollary \ref{dense-subspace-K-N}. Also observe that an application of Proposition \ref{prop2} to this case yields \cite[Theorem 5.1]{hadji}.

(2) $\displaystyle{\alpha_{n,k}=\frac{n-k+1}{n}}$: it corresponds to consider the sequence of Ces\`aro means $\left(M_n\right)_{n\geq 0}$ of the partial sum operators $S_n$:
\[
M_n=\frac{S_1+\ldots+S_n}{n}.
\]
In particular, if the sequence of operators $\left(S_n\right)_n$ is universal, then the sequence $\left(M_n\right)_n$ is also universal. This result is a version for generalized universal series of 
Ces\`aro hypercyclicity  (see \cite{Cost}).}
\end{example}

The examples given in this subsection are quite natural, but they live in $\mathbb{K}^{\mathbb{N}}$ 
(except Example \ref{applicationspropospacesequence} (2)), which is a strong limitation to produce interesting examples. In the next sections, we will exhibit examples of more sophisticated (restricted) generalized universal series, some living in spaces different from $\mathbb{K}^{\mathbb{N}}$, and inducing surprising properties (see Sections 4 and 5 especially).

\section{Examples of generalized universal series}\label{ExGe}
\subsection{Generalized Fekete universal series}
With the notations of the first section, the setting is the following. The space $E$ is the Fr\'echet space $\mathcal{C}^\infty(\mathbb{R})$ of $\mathcal{C}^\infty$ functions on $\mathbb{R}$, endowed with the topology defined by the countable family of semi-norms $\left(p_n\right)_n$, with
\[
p_n=\sum_{k=0}^n\sup_{\left[-n,n\right]}\left\vert f^{(k)}\right\vert.
\]
We take $A=\mathbb{R}^{\mathbb{N}}$ equipped with the Cartesian topology, and $X=\mathcal{C}_0\left(\mathbb{R}\right)$ the space of continuous functions on $\mathbb{R}$ which vanishes at $0$, endowed with the topology of uniform convergence on compact sets. We consider the sequences $\alpha\mathcal{X}:=\left(x_{n,k}\right)_{n\geq k\geq 0}=\left(\alpha_{n,k}x_k\right)_{n\geq k\geq 0}$, where $\alpha=\left(\alpha_{n,k}\right)_{n\geq k\geq 0}\subset \mathbb{K}$ is as in the previous section and $\mathcal{X}=\mathcal{F}=\left(x_k\right)_{k\geq 0}=\left(f_k\right)_{k\geq 0}:=\left(x^{k+1}\right)_{k\geq 0}$. Finally, the map $T_0$ is the Borel map which takes $f\in \mathcal{C}^\infty(\mathbb{R})$ to the sequences of its Taylor coefficients $\left(f^{(k)}(0)/k!\right)_{k\geq 0}$.

The following theorem is a generalization of Fekete's Theorem, from the point of view of $\mathcal{C}^{\infty}$ functions.
\begin{theorem}\label{fex}Let $\mu \subset\mathbb{N}$ be an increasing sequence. With the above notations, there exists a $\mathcal{C}^{\infty}$ function 
$f$ such that, 
for any continuous function $h:\mathbb{R}\rightarrow \mathbb{R},$ with $h(0)=0,$ 
any compact set $K\subset \mathbb{R},$ there exists an increasing sequence $\left(\lambda_n\right)_n \subset \mu$ such that 
\[
\sup_{x\in K}\left\vert \sum_{k=0}^{\lambda_n}\frac{f^{(k)}(0)}{k!}\alpha_{n,k}x^{k+1}-h(x)\right\vert\rightarrow 
0,\hbox{ as }n\rightarrow +\infty.
\]
In addition, if we denote by $\mathcal{U}^{\mu}\left(\alpha\mathcal{X}\right)\cap \mathcal{C}^{\infty}$ the set of such $\mu$-generalized Fekete universal functions, then this latter is a dense $G_\delta$ subset and contains, apart from $0$, a dense subspace and an infinite dimensional closed subspace of $\mathcal{C}^\infty(\mathbb{R})$.
\end{theorem}

\begin{proof}A combination of classical Fekete's result together with Proposition \ref{prop2} shows that $\mathcal{U}^{\mu}\left(\alpha\mathcal{X}\right)\cap {\mathbb{R}^{\mathbb{N}}}$ is non-empty, for every increasing sequence $\mu \subset\mathbb{N}$. Borel's Theorem ensures that $\mathcal{U}^{\mu}\left(\alpha\mathcal{X}\right)\cap \mathcal{C}^{\infty}$ is non-empty and the result follows from Theorem \ref{thm-gen-approx-lem} Part (A), except the very last assertion. Now an application of Theorem \ref{thm-gen-approx-lem} Part (A)-(3) ensures that $\mathcal{U}^{\mu}\left(\alpha\mathcal{X}\right)\cap \mathcal{C}^{\infty}\left([-1,1]\right)$ is spaceable. Finally, the conclusion follows from the proof of \cite[Theorem 3.14]{CMM}.
\end{proof}

\subsection{Bernstein generalized universal series} In this subsection, we investigate examples where the fixed 
sequence $(x_{n,k})_{n\geq k\geq 0}$ does not take the form $(\alpha_{n,k}x_k)_{n\geq k\geq n}$ with 
$\alpha_{n,k}\in\mathbb{K}.$ First we obtain a Fekete type result when we consider the sequence 
$(x_{n,k})_{n\geq k\geq 0}=(x^{k}(1-x)^{n-k})_{n\geq k\geq 0}.$ 

\begin{theorem}\label{B1} Let $\mu \in \mathbb{N}$ be an increasing sequence. There exists a sequence $a=(a_1,a_2,\dots)$ of real numbers such that, for every 
continuous real function $f$ on $[0,1]$ vanishing at $0,$ there exists a sequence 
$(\lambda_n) \subset\mathbb{N}$ of integers such that 
\[
\sum_{k=1}^{\lambda_n}a_k x^k(1-x)^{\lambda_n-k}\rightarrow f(x),\hbox{ as }n\rightarrow +\infty,
\hbox{ uniformly on }[0,1].
\]
The set of such $\mu$-generalized universal series is a dense $G_\delta$-subset of $\mathbb{R}^{\mathbb{N}}$ 
and contains, apart from $0$, a dense subspace of $\mathbb{R}^{\mathbb{N}},$ but it is not spaceable.
\end{theorem}

\begin{proof} Observe that for every integer $k\geq 1,$ $x^k(1-x)^{n-k}$ converges 
to $0$ uniformly on $[0,1],$ as $n\rightarrow +\infty.$ Let $\mu \subset\mathbb{N}$ be an increasing sequence. Then, according to Theorem \ref{thm-gen-approx-lem}, it suffices to prove that for every $\varepsilon>0,$ every $p\geq 1$ and every $h$ continuous on $[0,1]$ and vanishing at $0,$ one can find $n\in \mu$ and $a_p,\dots,a_n\in\mathbb{R}$ such that 
$\sup_{x\in [0,1]}\left\vert \sum_{k=p}^na_kx^k(1-x)^{n-k}-h(x)\right\vert <\varepsilon.$ \\
Since $h(0)=0$ and $h$ is a continuous function, there exists 
$\eta>0$ such that $\vert h(x)\vert<\varepsilon/3$ for $\vert x\vert<\eta.$ Define a continuous function 
$g$ on $[0,1]$ by 
\[
g(x)=\left\{\begin{array}{l}\displaystyle\frac{h(x)}{x^p}\hbox{ if }x\in[\eta,1]\\ \\
\displaystyle\frac{x}{\eta}\displaystyle\frac{h(x)}{\eta^p}\hbox{ if }x\in[0,\eta]\end{array}\right.
\]
Weierstrass' theorem ensures the existence of a polynomial $P(x)=\sum_{i=0}^lc_ix^i,$ with 
$l+p\in\mu$ (up to take $c_i=0$ for $d(P)+1\leq i\leq l$) satisfying 
\[
\sup_{x\in [0,1]}\vert P(x)-g(x)\vert<\varepsilon/3.
\]
We expand $P$ along the basis $x^k(1-x)^{l-k},$ $0\leq k\leq l,$ of polynomials of degree less than $l$ and write $P(x)=\sum_{i=0}^lb_ix^i(1-x)^{l-i}$. We set $n=l+p$ and we consider $x^pP(x)=\sum_{i=0}^lb_ix^{i+p}(1-x)^{l-i}=\sum_{k=p}^na_kx^{k}(1-x)^{n-k},$ 
with $a_k=b_{k-p}$ for $p\leq k\leq n$. Now observe that $x^pP(x)$ does the job. 
Indeed, for $0\leq x\leq \eta,$ we have 
\[
\begin{array}{rcl}\vert x^pP(x)-h(x)\vert&\leq &\vert x^p\vert \vert P(x)\vert+ \vert h(x)\vert\\
&\leq& \vert x^p\vert \vert P(x)-g(x)\vert +\vert x^p\vert \vert g(x)\vert + \vert h(x)\vert\\
&<& \varepsilon /3+ \vert h(x)\vert+\varepsilon/3<\varepsilon.\end{array}
\]
Now, for $x\in [\eta,1],$ we have 
\[
\vert x^pP(x)-h(x)\vert =\vert x^p\vert \vert P(x)-g(x)\vert<\varepsilon.
\]
The assertion is proved and we obtain the topological genericity as well as the existence of 
a dense universal subspace, except $0.$ Finally, using Proposition \ref{propospacesequence2} 
(see Examples \ref{applicationspropospacesequence} (4)), we conclude 
that the set of such generalized universal series, is not spaceable. 
\end{proof}

In the same spirit, we prove an analogue of Theorem \ref{B1} using classical Bernstein polynomials. 

\begin{proposition}\label{B2}Let $\mu \subset\mathbb{N}$ be an increasing sequence. There exists a sequence $a=(a_0,a_1,a_2,\dots)$ of real numbers such that, for every continuous real function $f$ on $(0,1),$ there exists an increasing sequence 
$(\lambda_n)\subset \mu$ such that 
\[
\sum_{k=0}^{\lambda_n-1}{\lambda_n \choose k}a_k x^k(1-x)^{\lambda_n-k}\rightarrow f(x),\hbox{ as }n\rightarrow +\infty,
\]
uniformly on every compact set in $(0,1).$ 
The set of such $\mu$-generalized universal series is a dense $G_\delta$ subset of $\mathbb{R}^{\mathbb{N}}$ and contains, apart from $0$, a dense subspace.  
\end{proposition}

\begin{proof} In order to apply the results of Section \ref{S2}, we have to check that for 
every integer $k\geq 0,$ the polynomials $B_{n,k}(x)={n \choose k}x^k(1-x)^{n-k}$ converges  
uniformly on every compact set in $(0,1).$ To see this, it suffices to observe that, for any integer $k\geq 0$ and 
any compact set $L\subset (0,1),$ 
there exists $n_0$ large enough such that for all $n\geq n_0,$ one has 
\[
\sup_{x\in L}\vert B_{n,k}(x)\vert\leq  B_{n,k}(1/\sqrt{n})\hbox{ and }B_{n,k}(1/\sqrt{n})\rightarrow 0,\hbox{ 
as }n\rightarrow +\infty.
\]
Let $\mu \subset\mathbb{N}$ be an increasing sequence. Now the proof of the proposition works as that of Theorem \ref{B1} with easy modifications. Indeed 
according to Theorem \ref{thm-gen-approx-lem}, it suffices to prove that for every $\varepsilon>0,$ every $p\geq 1,$ every continuous function $h$ on $(0,1)$ and every compact set $L\subset (0,1),$ one can find $n\in \mu$ and $a_p,\dots,a_n\in\mathbb{R}$ such that 
$\sup_{x\in L}\left\vert \sum_{k=p}^{n-1}{n\choose k}a_kx^k(1-x)^{n-k}-h(x)\right\vert <\varepsilon.$ \\
Weierstrass' theorem ensures the existence of a polynomial $P(x)=\sum_{i=0}^lc_ix^i,$ with $l+p+1\in\mu,$ (up to take $c_i=0$ for $d(P)+1\leq i\leq l$) satisfying 
\[
\sup_{x\in L}\vert P(x)-\frac{h(x)}{x^p(1-x)}\vert<\varepsilon/\inf_{x\in L}\vert x(1-x)\vert.
\]
We expand $P$ along the basis ${l+p+1\choose k+p}x^k(1-x)^{l-k},$ $0\leq k\leq l,$ of polynomials of degree less than $l$ and write $P(x)=\sum_{i=0}^l{l+p+1\choose i+p}b_ix^i(1-x)^{l-i}$. We set $n=l+p+1$ and we consider $x^p(1-x)P(x)=\sum_{i=0}^l{l+p+1\choose i+p}b_ix^{i+p}(1-x)^{l+1-i}=\sum_{k=p}^{n-1}
{n\choose k}a_kx^{k}(1-x)^{n-k},$ 
with $a_k=b_{k-p}$ for $p\leq k\leq n-1$. Now observe that 
$\sup_{x\in L}\left\vert x^p(1-x)P(x)-h(x)\right\vert <\varepsilon.$ 
\end{proof}

\begin{remark}  {\rm Let $a=(a_n)_{n\geq 0}$ be a universal sequence 
of real numbers such that the set of Bernstein sums $\sum_{j=0}^{p-1}{p \choose j}a_j x^j(1-x)^{p-j},$ 
for $p$ a 
prime number, is universal in the space of continuous functions on $(0,1)$ endowed with the topology 
of uniform convergence on compact sets (Proposition \ref{B2}). We can easily construct a non-continuous (but Lebesgue measurable) function $f$ such that the sequence of all its 
Bernstein sums $\left(\sum_{j=0}^{n}{n \choose j}f(\frac{j}{n}) x^j(1-x)^{n-j}\right)_n$ is  
universal in the same topological space. To do this, it suffices 
to set $f\left(\frac{j}{p}\right)=a_j,$ for $j=0,\dots, p-1,$ if $p$ is a prime number, and $f(t)=0$ otherwise (in fact $f$ is even equal to $0$ Lebesgue almost everywhere). In comparison, we recall that if $g$ is a continuous function on $[0,1]$ then $\sum_{j=0}^{n}{n \choose j}g(\frac{j}{n}) x^j(1-x)^{n-j}\rightarrow g(x),$ 
as $n\rightarrow +\infty,$ uniformly on $[0,1]$.
} 
\end{remark}

\medskip{}
In the context of measurable functions, Proposition \ref{B2} takes the following form. 

\begin{corollary}\label{B3} There exists a sequence $a=(a_0,a_1,a_2,\dots)$ of real numbers such that, for every 
$\sigma$-finite Borel measure $\mu$ on $[0,1]$ with $\mu(\{0\})=\mu(\{1\})=0$ and 
every $\mu$-measurable function $f$ on $[0,1],$ there exists a sequence 
$(\lambda_n)$ of integers such that 
\[
\sum_{k=0}^{\lambda_n-1}{\lambda_n \choose k}a_k x^k(1-x)^{\lambda_n-k}\rightarrow f(x),\hbox{ as }n\rightarrow +\infty,
\ \mu-\hbox{a.e.}
\]
The set of such sequences is a $G_\delta$ and dense subset of $\mathbb{R}^{\mathbb{N}}$ and contains, apart from $0$, a dense subspace.
\end{corollary}
\begin{proof} Since every $\mu$-measurable function is the almost everywhere pointwise limit of a sequence of continuous functions, Proposition \ref{B2} yields the result.
\end{proof}

The following result holds too. The proof works as that of Proposition \ref{B2} with obvious modifications.

\begin{proposition}\label{B3}Let $\mu \subset\mathbb{N}$ be an increasing sequence. There exists a sequence $a=(a_0,a_1,\dots)$ of real numbers such that, for every continuous real function $f$ on $(0,1],$ there exists an increasing sequence 
$(\lambda_n)\subset \mu$ such that 
\[
\sum_{k=0}^{\lambda_n}{\lambda_n \choose k}a_k x^k(1-x)^{\lambda_n-k}\rightarrow f(x),\hbox{ as }n\rightarrow +\infty,
\]
uniformly on every compact set in $(0,1].$ 
The set of such $\mu$-generalized universal series is a dense $G_\delta$ subset of $\mathbb{R}^{\mathbb{N}}$ and contains, apart from $0$, a dense subspace.  
\end{proposition}

\section{Generalized universal Taylor series on simply connected domains}\label{SectionBernstein}

In the whole section, $\Omega$ will denote a simply connected domain in $\mathbb{C}$. Let $\mathcal{H}\left(\Omega\right)$ be the Fr\'echet space of all holomorphic functions in $\Omega$, endowed with the translation-invariant metric $d_{\mathcal{H}\left(\Omega\right)}$, associated with the family of semi-norms $\left\Vert \cdot \right\Vert _n$ given by 
$\left\Vert f \right\Vert _n=\sup_{L_n}\left\vert f\right\vert,$ 
where $\left(L_n\right)_n$ is an exhaustion of compact subsets of $\Omega.$ 

Let us recall a result due to Nestoridis \cite{Nestor1} completed in the version below by \cite[Theorem 4.2]{Men}.

\begin{theorem}\label{thm-Nest} Let $\xi \in \Omega$. There exists a holomorphic function $f\in\mathcal{H}\left(\Omega\right)$ such that, for every compact subset $K\subset \mathbb{C}$, $K\cap \Omega=\emptyset$, $K^c$ connected, every function $h$ continuous on $K$ and holomorphic in the interior of $K$, there exists a sequence $\left(\lambda _n\right)_n$ such that
\[
\sup_{z\in K}\left\vert\sum_{k=0}^{\lambda _n}\frac{f^{(k)}\left(\xi \right)}{k!}\left(z-\xi \right)^k-h(z)\right\vert\rightarrow 0,\mbox{ as }n\rightarrow +\infty.
\]
Moreover, the set of such functions is a dense $G_{\delta}$ subset of $\mathcal{H}\left(\Omega\right)$ and contains, apart from $0$, a dense subspace and an infinite dimensional closed subspace of $\mathcal{H}\left(\Omega\right)$.
\end{theorem}

An crucial tool to prove the existence of such universal Taylor series is the following lemma, stated in this form in \cite{Nestor1}.

\begin{lemma}\label{lem-geo-melanes}There exists a sequence of compact sets $\left(K_n\right)_n\subset \mathbb{C}$, with $K_n\cap \Omega=\emptyset$ and $K_n^c$ connected for every $n\geq 0$, such that every compact set $K\subset \mathbb{C}$, with $K\cap \Omega=\emptyset$ and $K^c$ connected, is contained in some $K_n$, $n\geq 0$.
\end{lemma}

We will see that, in order to make the link with the formalism of Section \ref{S2}, it was necessary to state such a lemma from now.

In this section, we will be interested in restricted generalized universal series in the sense of Theorem \ref{thm-Nest}. But first of all, let us quote an easier result, which is a combination of Seleznev's Theorem together with the analogue of Lemma \ref{lem-geo-melanes}, where $\Omega$ is replaced by $\{\xi\}$, $\xi\in \mathbb{C}$ (\cite{luh}, see also \cite[Part B-1, Theorem 4]{bgnp}).

\begin{theorem}\label{selez}Let $\xi \in \mathbb{C}$. There exists a sequence $\left(a_n\right)_n\subset \mathbb{C}$ such that, for every compact set $K\subset \mathbb{C}\setminus \{\xi\}$, $K^c$ connected, every function $h$ continuous on $K$ and holomorphic in the interior of $K$, there exists an increasing sequence $\left(\lambda _n\right)_n$ such that
\[
\sup_{z\in K}\left\vert\sum_{k=0}^{\lambda _n}a_k\left(z-\xi\right)^k-h(z)\right\vert\rightarrow 0\mbox{ as }
n\rightarrow +\infty.
\]
\end{theorem}

For both Theorems \ref{thm-Nest} and \ref{selez}, the formalism of Section \ref{S2} is not general enough. 
We keep the notations of Section \ref{S2}, except that, instead of a single space $X$, we will consider 
a sequence of metrizable vector spaces $X_l$, $l\in \mathbb{N}$, with a translation-invariant metric $d_l$ respectively. 
Note that, in these two results, Mergelyan's Theorem ensure that it is sufficient to approach entire functions $h$ 
(in fact polynomials) on compact subsets. First of all, let $E=\mathcal{H}\left(\Omega\right)$ 
and $d_E=d_{\mathcal{H}\left(\Omega\right)}$. For any $l\geq 0$, let $X_l$ be the metric space $\mathcal{H}\left(\mathbb{C}\right)$ of entire functions endowed with the norm 
$\sup_{z\in K_l}\left\vert f\right\vert$, where $\left(K_l\right)_l$ is the sequence 
of compact sets in $\mathbb{C}$ given by Lemma \ref{lem-geo-melanes} (for Theorem \ref{selez}, 
$\Omega$ is replaced by $\{\xi\}$). Let also $\mathcal{X}=\mathcal{F}=\left(x_k\right)_{k\geq 0}
=\left(f_k\right)_{k\geq 0}=\left(\left(z-\xi\right)^k\right)_{k\geq 0}$ and $\alpha=
\left(\alpha_{n,k}\right)_{n\geq k\geq 0}$ be non-zero complex numbers such that 
$\left(\alpha_{n,k}\right)_{n}$ converges in $\mathbb{C}$ 
for every $k\geq 0$ (see Section \ref{S2}), and let $\alpha\mathcal{X}=\left(\alpha_{n,k}\left(z-\xi\right)^k\right)_{n\geq k\geq 0}$. 
$T_0$ is the map which takes $f\in \mathcal{H}\left(\Omega \right)$ to the sequence $\left(\frac{f^{(k)}(\xi)}{k!}\right)_{k\geq 0}\subset \mathbb{C}^{\mathbb{N}}$.

Now, Theorem \ref{selez} together with Proposition \ref{prop2}, Theorem \ref{thm-gen-approx-lem} Part B and 
Proposition \ref{propospacesequence2} (which both hold in the case of a countable sequence of spaces $\left(X_l\right)_l$) directly yield the following.
\begin{theorem}\label{gen-selez}With the notations above, for every $\xi\in \mathbb{C}$, there exists a sequence $\left(a_n\right)_n\subset \mathbb{C}$ such that, for every compact set $K\subset \mathbb{C}\setminus \{\xi\}$, $K^c$ connected, every function $h$ continuous on $K$ and holomorphic in the interior of $K$, there exists an increasing sequence $\left(\lambda _n\right)_n$ such that
\[
\sup_{z\in K}\left\vert\sum_{k=0}^{\lambda _n}a_k\alpha_{\lambda_n,k}\left(z-\xi\right)^k-h(z)\right\vert\rightarrow 0\mbox{, as }n\rightarrow +\infty.\]
The set of such sequences is a dense $G_{\delta}$ subset of $\mathbb{C}^{\mathbb{N}},$ contains, apart from $0$, a dense subspace of $\mathbb{C}^{\mathbb{N}}$ but it is not spaceable.
\end{theorem}

In this section, we will refer to such sequences as \emph{generalized universal Taylor series}. Even if we will omit to specify it, we may keep in mind that this notion makes sense with respect to the sequence $\left(\alpha_{n,k}\right)_{n\geq k\geq 0}$. The previous result asserts that there exists a generalized universal Taylor series with radius of convergence $0$. One of the main goal of this section is to prove that there also exist generalized universal Taylor series with positive radius of convergence, as soon as every sequence $\left(\alpha_{n,k}\right)_{n\geq 0}$, $k\geq 0$, does not converge to $0$ too fast.

We keep the notations introduced before Theorem \ref{gen-selez}, except that $E$ now stands for the space $\mathcal{H}\left(\Omega\right)$, that $\Omega$ is assumed to be bounded, and that the sequence $\left(K_n\right)_n$ is now given by Lemma \ref{lem-geo-melanes}. We introduce a condition, concerning the sequence $\left(\alpha_{n,k}\right)_{n\geq k\geq 0}$, which will be crucial in the sequel.

\noindent{}\textbf{Condition ($\hbox{C}_{\mu}$).} Let $\mu\subset\mathbb{N}$ be an increasing sequence. We say that $\left(\alpha_{n,k}\right)_{n\geq k\geq 0}$ satisfies Condition ($\hbox{C}_{\mu}$) if
\[
\limsup _{n\in \mu}\left(\min _{0\leq k\leq n}\left\{\sqrt[n]{\left\vert \alpha _{n,k}\right\vert }\right\}\right)\geq 1.
\]

\begin{remark}\label{rem-c-mu}{\rm Observe that $\left(\alpha_{n,k}\right)_{n\geq k\geq 0}$ satisfies Condition ($\hbox{C}_{\mu}$) if and only if for every $A>1$, there exists an increasing sequence $\nu (A)\subset \mu$ such that $\left\vert\alpha_{n,k}\right\vert\geq A^{-n}$ for every $n\in\nu$ and every $0\leq k\leq n$.}
\end{remark}

For reasons that will be clear in the sequel, we cannot obtain a result as general as we would wish. In order to deal with an arbitrary sequence $\left(\alpha_{n,k}\right)_{n\geq k\geq 0}$ satisfying Condition ($\hbox{C}_{\mu}$) for some $\mu$, we need to consider domains $\Omega$ which are discs. In order to work on an arbitrary bounded domain $\Omega$, we have to restrict ourselves to a class of sequences $\left(\alpha_{n,k}\right)_{n\geq k\geq 0}$ which has been considered in \cite{tsirivas}.

The notation $\mathbb{D}_R$ will stand for the disc of radius $R$ centered at $0$, and $\mathbb{D}_1$ will be simply denoted by $\mathbb{D}$. The main result of this section states as follows.
\begin{theorem}\label{gen-Tayl-series} Let $\Omega$ be a disc, let $\xi \in \Omega$, let $\mu\subset\mathbb{N}$ be an increasing sequence and let $\left(\alpha_{n,k}\right)_{n\geq k\geq 0}$ be a sequence of non-zero complex numbers satisfying Condition ($\hbox{C}_{\mu}$). There exists a function $f\in \mathcal{H}\left(\Omega\right)$ such that, for every compact subset $K$ of $\mathbb{C}$, with $K\cap \Omega=\emptyset$ and $K^c$ connected, and every function $h$ continuous on $K$ and holomorphic in the interior of $K$, there exists an increasing sequence $\left(\lambda_n\right)_n\subset \mu$ such that
\[
\sup_{z\in K}\left\vert\sum_{k=0}^{\lambda _n}\frac{f^{(k)}\left(\xi\right)}{k!}\alpha_{\lambda_n,k}\left(z-\xi\right)^k-h(z)\right\vert\rightarrow 0\mbox{, as }n\rightarrow +\infty.
\]
We denote by $\mathcal{U}_{\alpha}^{\mu}\left(\Omega ,\xi \right)$ the set of such functions. Moreover, if $\left(\alpha_{n,k}\right)_{n\geq k\geq 0}$ satisfies Condition ($\hbox{C}_{\mu}$) for every $\mu\subset\mathbb{N}$, then $\mathcal{U}_{\alpha}^{\mu}\left(\Omega ,\xi \right)$ is a dense $G_{\delta}$ subset of $\mathcal{H}\left(\Omega\right)$ and contains, apart from $0$, a dense subspace and an infinite dimensional closed subspace of $\mathcal{H}\left(\Omega\right)$, for every $\mu$.
\end{theorem}

The proof of this result appeals to a quantitative version of Runge's Theorem, that is Bernstein-Walsh's Theorem, and an easy Bernstein's type inequality. For the first one, we need to introduce some background. We will say that a compact $K\subset \mathbb{C}$, with connected complement, is regular for the Dirichlet problem at infinity if there exists a function $g_K$, continuous on $\mathbb{C}$, harmonic in $\mathbb{C}\setminus K$, equal to $0$ on $K$ (equal to $0$ on the boundary of $K$ and then extended by $0$ on $K$), and equal, up to a bounded function, to $\log\left\vert z\right\vert$ near infinity. If it exists, this function (that we will always denote by $g_K$) is unique and we will refer to it as the solution of the Dirichlet problem at infinity for $K$. It can be shown that $g_K$ does not vanish in $\mathbb{C}\setminus K$. In particular, when $K$ is the closed disc of radius $R$, centered at $0$, $g_K$ is equal to $\log\vert z\vert-\log R$. We now quote Bernstein-Walsh's Theorem as it appears in \cite[Theorem 1.1]{bagby} (see also \cite[Chapter IV]{walsh}).

\begin{theorem}\label{BWT}[Bernstein-Walsh's Theorem]Let $K$ be a compact subset of $\mathbb{C}$, with connected complement, which is regular for the Dirichlet problem. Let $f$ be a function bounded on $K$, and for each nonnegative integer $n$ set
\[
D_n\left(f,K\right)=\inf\left\{\sup_K\left\vert f-P\right\vert:\,P\mbox{\emph{ is a holomorphic polynomial of degree less than }}n\right\}.
\]
Let $0\leq\varrho<1$. Then $\limsup_{n\rightarrow\infty}D_n\left(f,K\right)^{1/n}\leq \varrho$ if and only if $f$ is the restriction to $K$ of a function holomorphic on $\left\{z\in \mathbb{C}:\,g_K(z)<\log 1/\varrho\right\}$, where $g_K$ is the Green function for $K$.
\end{theorem}
Note that every compact subset of $\mathbb{C}$, with a finite number of connected components, connected complement and $\mathcal{C}^{\infty}$-boundary, is regular for the Dirichlet problem.

We also need the following result. This is a kind of Bernstein inequality.
\begin{lemma}\label{prop-Bern-Ineq}Let $R>0$ be a real number. Let $P(z)=\sum_{k=0}^na_kz^k$ be a (holomorphic) polynomial and let $\alpha=\left(\alpha _0,\ldots,\alpha _{n}\right)$ be $n+1$ complex numbers. We denote by $P_\alpha$ the polynomial $P_{\alpha}(z)=\sum_{k=0}^n\alpha _ka_kz^k$. For every compact set $K\subset \mathbb{D}_R$, we have
\begin{equation}\label{Bernstein-Ineq}\sup_{z\in K}\left\vert P_\alpha (z)\right\vert\leq \sqrt{\frac{2}{d_K^{\mathbb{D}_R}}}\max_{0\leq k\leq n}\left\vert\alpha _k\right\vert\sup_{z\in \mathbb{D}_R}\left\vert P(z)\right\vert,
\end{equation}
where $d_K^{\mathbb{D}_R}$ stands for the distance from $K$ to the complement of $\mathbb{D}_R$.
\end{lemma}

\begin{proof}Let $K\subset \mathbb{D}_R$ be a compact set. Without loss of generality, we may and shall assume that $K$ is a disc of radius $0<r<R$, centered at $0$. Let $\displaystyle{r_0=\frac{r+R}{2}}$. For any $s>0$, let $H^2\left(\mathbb{D}_s\right)$ be the Hardy space of the disc $\mathbb{D}_s$, that is
\[
H^2\left(\mathbb{D}_s\right)=\left\{f\in \mathcal{H}\left(\mathbb{D}_s\right),\,\sup_{0<t<s}\int_{\mathbb{T}}\left\vert f(tz)\right\vert^2 \emph{d}\sigma (z)<+\infty\right\},
\]
where $\sigma$ is the rotation-invariant measure on the unit circle $\mathbb{T}$, endowed with the norm $\left\Vert \cdot\right\Vert _{2,s}=\sqrt{\sup_{0<t<s}\int_{\mathbb{T}}\left\vert f(tz)\right\vert^2 \emph{d}\sigma (z)}$. Using that, for any $f(z)=\sum_{k=0}^{+\infty}a_kz^k\in H^2\left(\mathbb{D}_s\right)$, $\left\Vert f\right\Vert^2_{2,s}=\left\Vert f_s\right\Vert^2_{2,1}=\sum_{k=0}^{+\infty}\left\vert a_k\right\vert^2s^{2k}$, with $f_s(z)=f(sz)$ for every $z\in \mathbb{D}$, we get
\begin{eqnarray*}\left\Vert P_{\alpha}(z)\right\Vert_{2,r_0}^2 & = & \left\Vert  \left(P_{\alpha}\right)_{r_0}(z)\right\Vert_{2,1}^2\\
& = & \sum_{k=0}^n\left\vert\alpha _ka_k\right\vert^2r_0^{2k} \\
& \leq & \max_{0\leq k\leq n}\left\vert\alpha _k\right\vert^2\sum_{k=0}^n\left\vert a_k\right\vert^2R^{2k} \\
& = & \max_{0\leq k\leq n}\left\vert\alpha _k\right\vert^2\left\Vert P\right\Vert^2_{2,R}\\
& \leq & \max_{0\leq k\leq n}\left\vert\alpha _k\right\vert^2\sup_{z\in \mathbb{D}_R} \left\vert P(z)\right\vert^2.
\end{eqnarray*}
Now, from Cauchy's formula and the previous estimate, it follows
\begin{eqnarray*}\sup_{z\in \mathbb{D}_r}\left\vert P_{\alpha}(z)\right\vert^2 & \leq & \frac{2}{R-r}\max_{0\leq k\leq n}\left\vert\alpha _k\right\vert^2\sup_{z\in \mathbb{D}_R}\left\vert P(z)\right\vert^2,
\end{eqnarray*}
as desired.
\end{proof}

\begin{remark}{\rm Up to take $r_0=\lambda R+(1-\lambda)r$ instead of $(R+r)/2$, and letting $\lambda$ tend to $1$ in the proof of the previous proposition, the constant $\sqrt{2/d_K^{\mathbb{D}_R}}$ can be replaced by $\sqrt{1/d_K^{\mathbb{D}_R}}$. Anyway, this will not be of importance for us.}
\end{remark}

We can now turn to the proof of Theorem \ref{gen-Tayl-series}.
\begin{proof}[of Theorem \ref{gen-Tayl-series}]Up to a translation, we may and shall assume that $\Omega$ is a disc centered at $0$ and that $\xi =0$. Then we can also assume that each compact set $L_n$ in the exhaustion $\left(L_n\right)_n\subset \Omega$ is a disc centered at $0$. Let $\mu \subset\mathbb{N}$ be an increasing sequence. By Lemma \ref{lem-geo-melanes} and Theorem \ref{thm-gen-approx-lem} Part B, we are reduced to prove the following: Let $\varepsilon>0$, let $h$ be an entire function, let $K\subset \mathbb{C}\setminus \Omega$ be a compact set with connected complement, and let $L\subset \Omega$ be a disc centered at $0$. There exists an integer $n>0$ and a (holomorphic) polynomial $P$ with degree less than $\mu_ n$ such that
\[
\sup_K\left\vert h(z)-S^{\alpha\mathcal{X}}_{\mu_ n}\left(P\right)\right\vert\leq \varepsilon\hbox{ and }\sup_L\left\vert P\right\vert\leq \varepsilon,
\]
where $\alpha=\left(\alpha _{n,k}\right)_{0\leq k\leq n}$ is as in the statement of the theorem.
Let us prove this fact. Since $K$ and $L$ are disjoint compact subsets, there exists $\eta>0$ such that, if we set
\[
K_{\eta}=\left\{z\in \mathbb{C}:\,d\left(z,K\right)\leq \eta\right\}\hbox{ and }L_{\eta}=\left\{z\in \mathbb{C}:\,d\left(z,L\right)\leq \eta\right\},
\]
then $K_{\eta}$ and $L_{\eta}$ are compact, and $K_{\eta}\cap L_{\eta}=\emptyset$. $L_{\eta}$ is a disc centered at $0$ that we will simply denote by $L'$. In addition, since every connected component $C$ of $K_{\eta}$ contains a disc of radius $\eta/2$, $K_{\eta}$ has a finite number of connected components. Let us denote by $c$ this number. Now, by a compactness argument, $K_{\eta}\setminus K$ have positive width in the sense that there exists $c$ closed polygonal Jordan curves $\Gamma$ in the interior of $K_{\eta}\setminus K$, with a finite number of vertices, such that there exists $r>0$ such that, for any $\zeta \in \Gamma$, the disc of center $\zeta$ and radius $r$ is contained in the interior of $K_{\eta}\setminus K$. Now, since every polygonal Jordan curve with a finite number of vertices can be uniformly approximated by a $\mathcal{C}^{\infty}$-Jordan curve, it follows that $K_{\eta}\setminus K$ contains $c$ $\mathcal{C}^{\infty}$-Jordan curves, which are the boundaries of the finitely many connected components of a simply connected domain $K'$, with connected complement, containing $K$. In particular, $K'\cup L'$ is a domain which satisfies the assumption of Bernstein-Walsh's Theorem. Up to take an open neighborhood $U$ of $K'\cup L'$ small enough, we may assume that $h_0$ is the restriction to $K'\cup L'$ of a holomorphic function on $U$. Now, let $g_{K'\cup L'}$ be the solution ot the Dirichlet problem at infinity for $K'\cup L'$. Since $g_{K'\cup L'}$ is continuous on $\mathbb{C}$ and does not vanish in the complement of $K'\cup L'$, there exists $0\leq \varrho<1$ such that $\left\{z\in \mathbb{C}:\,g_K(z)<\log 1/\varrho\right\}\subset U$. Let $\varrho'$ and $A$ be such that $0\leq \varrho<\varrho'<A^{-1}<1$. Since $\left(\alpha_{n,k}\right)_{n\geq k\geq 0}$ satisfies Condition ($\hbox{C}_{\mu}$) and according to Remark \ref{rem-c-mu}, there exists $\nu\subset \mu$ such that $\left\vert\alpha_{n,k}\right\vert\geq A^{-n}$ for $n\in\mathbb{N}u$ large enough and every $0\leq k\leq n$. Therefore, we can apply Bernstein Walsh's Theorem to $K'\cup L'$ and $h_0$, to get the existence of an increasing sequence $\left(\nu_ {n_k}\right)_{k\geq 0}\subset \nu$ and a sequence of polynomials $\left(P_{\nu_ {n_k}}\right)_{k\geq 0}$, $P_{\nu_ {n_k}}$'s being of degree less than $\nu_ {n_k}$ respectively, such that, for every $k\geq 0$,
\begin{equation}\label{B-W-appli}\sup_{z\in K'}\left\vert h(z)-P_{\nu_ {n_k}}\right\vert\leq \left(\varrho'\right)^{n_k}\hbox{ and }\sup_{z\in L'}\left\vert P_{\nu_ {n_k}}\right\vert\leq \left(\varrho'\right)^{\nu_ {n_k}}.
\end{equation}
Let us write $P_{\nu_ {n_k}}=\sum _{i=0}^{\nu_ {n_k}}a_{k,i}z^i$ for any $k\geq 0$ (if the degree of $P_{\nu_ {n_k}}$ is less than $\nu_ {n_k}$, then we complete the sum up to $\nu_ {n_k}$ by $0$), and define $Q_{\nu_ {n_k}}=\sum _{i=0}^{\nu_ {n_k}}\frac{a_{k,i}}{\alpha _{\nu_ {n_k},i}}z^i$. By (\ref{B-W-appli}), we get
\[
\sup_{z\in K'}\left\vert h(z)-S^{\alpha\mathcal{X}}_{\nu_ {n_k}}\left(Q_{\nu_ {n_k}}\right)\right\vert=\sup_{z\in K'}\left\vert h(z)-P_{\nu_ {n_k}}\right\vert\leq \varepsilon,
\]
for $k$ large enough. It remains to show that $\sup_L\left\vert Q_{\nu_ {n_k}}\right\vert\leq \varepsilon$ for $k$ large enough. With the notations of Lemma \ref{prop-Bern-Ineq}, we have $Q_{\nu_ {n_k}}=\left(P_{\nu_ {n_k}}\right)_{\beta}$ for any $\nu_ {n_k}$, with $\beta$ the sequence $\left(1/\alpha _{\nu_ {n_k},i}\right)_{0\leq i\leq \nu_ {n_k}}$ (observe that Condition $\hbox{C}_{\mu}$ ensures that the $\alpha _{\nu_ {n_k},i}$'s are non-zero). Since $L$ is a compact subset of the interior of the disc $L'$, we can apply Lemma \ref{prop-Bern-Ineq} to every $P_{\nu_ {n_k}}$, $k\geq 0$, and to $\beta$, $L'$ and $L$ to obtain
\begin{equation}\label{sup-L}\sup_L\left\vert Q_{\nu_ {n_k}}\right\vert\leq C\max_{0\leq i\leq \nu_ {n_k}}\left(\frac{1}{\left\vert \alpha _{\nu_ {n_k},i}\right\vert} \right)\sup_{L'}\left\vert P_{\nu_ {n_k}}\right\vert\leq C\max_{0\leq i\leq \nu_ {n_k}}\left(\frac{1}{\left\vert \alpha _{\nu_ {n_k},i}\right\vert}\right)\left(\varrho'\right)^{\nu_ {n_k}},
\end{equation}
for every $k\geq 0$, where $C$ is a constant independent of $k$. Now, it follows from Condition $\hbox{C}_{\mu}$ that
\[
\max_{0\leq i\leq \nu_ {n_k}}\left(\frac{1}{\left\vert \alpha _{\nu_ {n_k},i}\right\vert }\right)\left(\varrho'\right)^{\nu_ {n_k}}\leq \left(A\varrho'\right)^{\nu_ {n_k}},
\]
for every $0\leq i\leq \nu_ {n_k}$. Since $0<A\varrho'<1$, $\left(A\varrho'\right)^{\nu_ {n_k}}$ tends to $0$ as $k$ tends to $+\infty$, which gives the conclusion by (\ref{sup-L}).
\end{proof}

\begin{remark} {\rm Note that the sequence $\left(\alpha_{n,k}\right)_{n\geq k\geq 0}$ satisfies Condition $\hbox{C}_{\mu}$ for every $\mu\subset\mathbb{N}$ if and only if for every $A>1$, $\left\vert\alpha_{n,k}\right\vert\geq A^{-n}$ for every $n\in\mathbb{N}$ large enough and every $0\leq k\leq n$.}
\end{remark}
We deduce from Theorem \ref{gen-Tayl-series} and the previous remark the following corollary.
\begin{corollary}\label{coro-Tay-Bern} Under the notations and assumptions of Theorem \ref{gen-Tayl-series}, if for every $A>1$, $\left\vert\alpha_{n,k}\right\vert\geq A^{-n}$ for every $n\in\mathbb{N}$ large enough and every $0\leq k\leq n$, then for every $\mu$, $\mathcal{U}_{\alpha}^{\mu}\left(\Omega ,\xi \right)$ contains, apart from $\{0\}$, a dense subspace and an infinite dimensional closed subspace of $\mathcal{H}\left(\Omega \right)$.
\end{corollary}
Theorem \ref{gen-Tayl-series} shows that if the sequence $\left(\alpha_{n,k}\right)_{n\geq k\geq 0}$ does not go to $0$ faster than exponentially along some subsequences, then $\mathcal{U}_{\alpha}\left(\Omega ,\xi \right)$ is not empty. The question is whether the converse is true. The following proposition approaches a positive answer.

\begin{proposition}\label{propopt} Let $\Omega\subset \mathbb{C}$ be a simply connected domain, let $\xi \in \Omega$ and let $\mu\subset\mathbb{N}$ be an increasing sequence. With the notations of Theorem \ref{gen-Tayl-series}, if $\mathcal{U}_{\alpha}^{\mu}(\Omega,\xi)\ne\emptyset$ then
\begin{equation}\label{eqoptopt}\limsup _{n\in \mu}\left(\max _{0\leq k\leq n}\left\{\sqrt[n]{\left\vert \alpha _{n,k}\right\vert }\right\}\right)\geq 1.
\end{equation}
\end{proposition}


\begin{proof} Let $D_{r}(\xi)$ be the largest open ball centered at $\xi$, contained in $\Omega$. Let $x\in \overline{D_{r}(\xi)}\cap \partial{\Omega}$, where $\partial{\Omega}$ stands for the boundary of $\Omega$. Let us write $x=\xi +re^{i\vartheta}$ with $\vartheta \in [0,2\pi[$. Since $\Omega$ is simply connected, we can find $\varepsilon _n\rightarrow 0$ and $\vartheta _n\rightarrow \vartheta$, such that every point $x_{n}=\xi +r\left(1+\varepsilon_{n}\right)e^{i\vartheta _n}$ is contained in $\Omega ^c$.

Let $f\in \mathcal{H}(\Omega),$ with $f(z)=\sum_{k\geq 0}a_{k}\left(z-\xi\right)^{k}$ for any $z\in D_{r}(\xi)$ 
and assume that $f\in \mathcal{U}_{\phi}^{\mu}(\Omega,\xi)$. Then, approximating $1$ uniformly on the compact subset $\{x_n\}\subset \Omega ^c$ by some subsequences of $\left(\sum_{k=0}^ma_{k}\alpha _{n,k}\left(z-\xi\right)^{k}\right)_{m}$, we get an increasing subsequence  $(\nu _j)_{j}\subset \mu$, so that 
\[
\left\vert 1-\sum_{k=0}^{\nu _j}a_{k}\alpha _{\nu _j,k}r^{k}(1+\varepsilon_{n})^{k}e^{ik\vartheta _{n}}\right\vert\leq \frac{1}{2}.
\]
By the triangle inequality we obtain, for every $n,j\in\mathbb{N},$ 
\[
\sum_{k=0}^{\nu _j}\left\vert \alpha _{\nu _j,k}a_{k}\right\vert r^{k}\left(1+\varepsilon_{n}\right)^{k} \geq \frac{1}{2}.
\]
Furthermore, $\sum_{k\geq 0}a_k\left(z-\xi\right)^k\in \mathcal{H}\left(D_r(\xi)\right)$, so that for $k$ large enough,
\[
\left\vert a_k\right\vert \leq \frac{1}{r^k\left(1-\varepsilon _n\right)^k}.
\]
Hence, for every $n,j\in\mathbb{N}$ we have
\[
\frac{1}{2}\leq \sum_{k=0}^{\nu_j} \left\vert \alpha _{\nu _j,k}a_{k}\right\vert r^{k}\left(1+\varepsilon_{n}\right)^{k}\leq \max _{0\leq k\leq \nu _j}\left(\left\vert \alpha _{\nu _j,k}\right\vert\right) \left(\nu_j+1\right)\left(\frac{1+\varepsilon_n}{1-\varepsilon _n}\right)^{\nu_j}.
\]
Thus, for any $n\in\mathbb{N},$ there exists $\nu\subset \mu$ such that 
$\displaystyle{\max_{0\leq k\leq m}\left(\left\vert \alpha _{m,k}\right\vert\right)
\geq \frac{1}{2\left(m+1\right)}\left(
\frac{1+\varepsilon_n}{1-\varepsilon_n}\right)^{-m}}$ for every $m\in \nu$. Since $\varepsilon_n$ tends to $0$ as $n$ tends to infinity, Condition (\ref{eqoptopt}) follows.
\end{proof}

The following corollary summarizes the previous main results of the present section.

\begin{corollary}\label{coro-summ-BW}Let $\Omega\subset \mathbb{C}$ be a simply connected domain, let $\xi \in \Omega$ and let $\mu\subset\mathbb{N}$ be an increasing sequence.
\begin{enumerate}\item Assume that 
the sequence $\left(\alpha _{n,k}\right)_n$ is convergent in $\mathbb{C}$ for every $k\geq 0.$ 
If $\Omega$ is a disc and if we have 
\begin{equation}\label{eq-coro-1}\limsup _{n\in \mu}\left(\min _{0\leq k\leq n}\left\{\sqrt[n]{\left\vert \alpha _{n,k}\right\vert }\right\}\right)\geq 1,\end{equation}
then $\mathcal{U}_{\alpha}^{\mu}(\Omega,\xi)\ne\emptyset$;
\item If we have $\mathcal{U}_{\alpha}^{\mu}(\Omega,\xi)\ne\emptyset$ then
\begin{equation}\label{eq-coro-2}\limsup _{n\in \mu}\left(\max _{0\leq k\leq n}\left\{\sqrt[n]{\left\vert \alpha _{n,k}\right\vert }\right\}\right)\geq 1.\end{equation}
\end{enumerate}
\end{corollary}

\noindent{}\textbf{Important remark.} The previous results (Corollary \ref{coro-summ-BW}) are not totally satisfying for at least two reasons:

(i) Firstly, Equations (\ref{eq-coro-1}) and (\ref{eq-coro-2}) are of course not equivalent in general. The first one comes from Lemma \ref{prop-Bern-Ineq} again, and it seems to be difficult to sharpen Estimate (\ref{Bernstein-Ineq}) by replacing $\max \left(\alpha _k\right)$ by something smaller. Equation (\ref{eq-coro-2}) seems to be optimal if no extra assumption is made on the sequence $\left(\alpha_{n,k}\right)_{n\geq k\geq 0}$.

(ii) Secondly, Point (1) of Corollary \ref{coro-summ-BW} holds only when $\Omega$ is a disc. Actually, Estimate (\ref{sup-L}) in the proof of Theorem \ref{gen-Tayl-series} is given by Lemma \ref{prop-Bern-Ineq}, a proof of which for general domains is rather unlikely. In fact, we can show that an estimate similar to (\ref{Bernstein-Ineq}) holds in general, but with constant depending on the shape and the size (the ``distortion'') of the compact set together with the degree of the polynomial. This is not sharp enough to ensure that $\sup_L\left\vert Q_{\nu _{n_k}}\right\vert$ is small in the proof of Theorem \ref{gen-Tayl-series}. Thus we can ask the following question:\\

\noindent{}\textbf{Question.} Does Theorem \ref{gen-Tayl-series} hold for an arbitrary simply connected domain $\Omega$?

\medskip{}

Nevertheless, there is a particular case in which Restrictions (i) and (ii) are no more necessary: until the end of the present remark, let assume that the sequence $\left(\alpha _{n,k}\right)_{n\geq k\geq 0}$ does not depend on $k$ anymore, and write $\alpha _{n,k}=1/\phi(n)$ for every $k\geq 0$, and for some $\phi:\mathbb{N}\rightarrow \mathbb{C}$ convergent in $\mathbb{C}\setminus \{0\}\cup \{\infty\}$ (see Example \ref{1st-exs-gen} (1)).

We now make the following observations.

(i) Equations (\ref{eq-coro-1}) and (\ref{eq-coro-2}) above are trivially equivalent, and since $\phi(n)$ is convergent in $\mathbb{C}\setminus \{0\}\cup \{\infty\}$, they are also equivalent to:
\begin{equation}\label{eq-rem-bis}\liminf _{n\in \mu}\sqrt[n]{\left\vert \phi(n)\right\vert }=1.
\end{equation}

(ii) With the notations of the proof of Theorem \ref{gen-Tayl-series}, the polynomial $Q_{\nu _{n_k}}$ is equal to $\phi(\nu _{n_k})P_{\nu _{n_k}}$. Then, whatever the domain $\Omega$ and the compact subset $L$ inside $\Omega$, Estimate (\ref{sup-L}) is trivially true, i.e. $Q_{\nu _{n_k}}$ is small uniformly in $L$, whenever $\phi$ satisfies Condition (\ref{eq-rem-bis}). Therefore, Theorem \ref{gen-Tayl-series} holds in this case for every domain $\Omega$.

In this particular case, the method developed in this paper (via Theorem \ref{thm-gen-approx-lem}) allows us to recover the main result of \cite{tsirivas} when $\phi(n)$ is convergent in $\mathbb{C}\setminus \{0\}\cup\{\infty\}$. Here, we simply denote by $\mathcal{U}_{\phi}(\Omega,\xi)$ the set $\mathcal{U}_{\alpha}(\Omega,\xi)$.

\begin{theorem}\label{gen-Tayl-series-phi}Let $\Omega$ be a simply connected domain in $\mathbb{C}$, and let $\xi \in \Omega$ and let $\phi:\mathbb{N}\rightarrow \mathbb{C}$ convergent in $\mathbb{C}\setminus \{0\}\cup \{\infty\}$. The set 
$\mathcal{U}_{\phi}(\Omega,\xi)$ is non-empty if and only if
\[
\liminf _{n\in \mu}\sqrt[n]{\left\vert \phi(n)\right\vert }=1.
\]
\end{theorem}

\begin{remark}{\rm To finish, we mention that the generalized universal Taylor series given by Theorem \ref{gen-Tayl-series} possess Ostrowski-gaps. Let us recall that for a strictly increasing subsequence
$(n_k)$ in $\mathbb{N},$ we say that a power
series $\sum_{\nu=0}^{+\infty}a_{\nu}(z-\xi)^{\nu}$ possesses Ostrowski-gaps with respect to $(n_k)$
if there are $\epsilon_k>0$, $k\geq 1$, so that $\lim_{\nu\in I}\vert a_{\nu}\vert^{1/\nu}=0,$ for
$I=\cup_{k\geq 1}[\epsilon_kn_k,n_k].$ Assume that $(\alpha_{n,k})_{0\leq k\leq n}$ satisfies 
Condition ($\hbox{C}_\mu$) for every $\mu.$ According to Corollary \ref{coro-Tay-Bern}, for every $\mu$, there exists 
$f\in \mathcal{U}_{\alpha}^{\mu}\left(\Omega ,\xi \right).$ Up to take a translation, 
assume that $\xi=0.$
By using the universality, we get a sequence $S_{n_{k}}^{\alpha\mathcal{X}}(f)(z)$ which simply converges on a closed and non-thin set at $\infty$ (see for instance \cite[Section 2]{MY} for the definition).
Thus, by using \cite[Lemma 2]{MY}, we get for every $R>0$:
\[
\limsup_{k \rightarrow +\infty} \left(\sup_{ z \in \overline{D(0,R)}}  \left\vert S_{n_{k}}^{\alpha\mathcal{X}}(f)(z)\right\vert \right)^{1/n_{k}}\leq 1.
\]
Up to extract a subsequence from $(n_{k})$, we get for $k$ large enough:
\[
\sup_{z \in \overline{D(0,k)}}\left\vert \sum_{j=0}^{n_{k}}\alpha_{n_{k},j}a_{j}z^{j}\right\vert \leq 2^{n_{k}}.
\]
Now by Cauchy's formula, $\displaystyle \left\vert \alpha_{n_{k},j}a_{j}k^{j}\right\vert \leq 2^{n_{k}}$ for every $0\leq j\leq n_k$, and since $\alpha_{n_{k},j}\geq A^{-n_k}$ for every $A>1$ and $k\in\mathbb{N}$ large enough, we get
\[
\displaystyle \left\vert a_{j}\right\vert\leq \left(\frac{2A}{e}\right)^{\log (k)}\rightarrow 0,\hbox{ as }
k\rightarrow +\infty,
\]
for every $j \in \left[\frac{n_{k}}{\log(k)},n_{k}\right]$ and some $1<A<e/2.$ 
Hence, we get the existence of Ostrowski-gaps $\left[\frac{n_{k}}{\log(k)},n_{k}\right]$ for $S_{N}^{\alpha\mathcal{X}}(f)(z)$. Moreover, by using Theorem 1 of \cite{Luh2}, we also get in the case of 
$\alpha_{n,k}=1/\phi(n)$ that a generalized universal series with respect
to the center $\xi_1$ is also a generalized universal series with respect
to the center $\xi_2,$ for $\xi_1,\xi_2\in\Omega$ (see \cite{melanes1} for the classical case).}
\end{remark}

\section{Derivatives and generalized universal Taylor series}\label{deriv-Section5}
As in the previous section, we denote by $\mathbb{D}_r=\{z\in\mathbb{C}:\ \vert z\vert<r\}$ the disc centered at $0$, of radius $r\in\mathbb{R}_+\cup\{+\infty\}$, with the convention $\mathbb{D}_{\infty}=\mathbb{C}$. The notation $\mathcal{H}\left(\mathbb{D}_0\right)$ will stand for the space of formal power series, i.e. $\mathbb{C}^{\mathbb{N}}$ endowed with the Cartesian topology. As usual, for $f$ holomorphic around $0$, $S_n(f)(z)$ stands for the $n$-th partial sum of the Taylor development of $f$ at $0$.

The main result of this section states as follows.

\begin{theorem}\label{gts} Let $\mu=(\mu_n)_{n\geq 0}$ be an increasing sequence of positive integers. 
Let $(\alpha_n)_{n\geq 0}$ be a sequence of non-zero complex numbers and $R\in \left[0,+\infty\right]$ be the radius of convergence of the power series $\sum_{n\geq 0}\frac{z^n}{n!\vert \alpha_n\vert}.$ There exists a function $f\in \mathcal{H}\left(\mathbb{D}_R\right)$ such that, for every compact subset $K$ of $\mathbb{C}$, with $K^c$ connected, and every function $h$ continuous on $K$ and holomorphic in the interior of $K$, there exists an increasing subsequence $\left(\lambda_n\right)_n\subset \mu$ such that
\[
\sup_{z\in K}\left\vert \alpha_{\lambda _n} S_{\lambda _n}(f^{(\lambda _n)})(z)-h(z)\right\vert\rightarrow 0\hbox{ as }n\rightarrow +\infty.
\]
Moreover, the set $\mathcal{W}_R^{(\mu)}(\alpha_n)$ of such functions is a dense $G_{\delta}$ subset of $\mathcal{H}\left( \mathbb{D}_R\right)$ and contains a dense subspace, apart from $0,$ and an infinite dimensional closed subspace of $\mathcal{H}\left(\mathbb{D}_R\right)$.
\end{theorem}

Theorem \ref{gts} provides a large class of examples, some being rather surprising:

\begin{example}\rm{ (1) Taking $\alpha _n=1$ for any $n\geq 0$, Theorem \ref{gts} ensures that there exists an entire function $f$ which is universal in the sense that the set $\left\{S_n\left(f^{(n)}\right),\,n\in\mathbb{N}\right\}$ is dense in $\mathcal{H}\left(\mathbb{C}\right)$. In addition, the set of such functions is both topologically and algebraically generic and it is 
spaceable. It reminds MacLane's result on the existence of an entire function $f$ such that for every entire 
function $g$ there exists an increasing sequence $(\lambda_n)$ of positive integers such that 
$f^{(\lambda_n)}(z)\rightarrow g(z)$ locally uniformly in $\mathbb{C},$ as $n\rightarrow +\infty$ 
\cite{MacLane}.\\
(2) Taking $\alpha _n=n !$ for any $n\geq 0$, Theorem \ref{gts} ensures that there exists an holomorphic function 
$f\in \mathcal{H}(\mathbb{D})$ which is universal in the sense that the set $\left\{\frac{1}{n!}S_n\left(f^{(n)}\right),\,n\in\mathbb{N}\right\}$ is dense in $\mathcal{H}\left(\mathbb{C}\right)$. In addition, the set of such functions is both topologically and algebraically generic and it is spaceable. \\
(3) More generally, with a good choice of $\left(\alpha _n\right)_{n\geq 0}$, we can show that there exists a power series with arbitrary radius of convergence which is universal in the sense of Theorem \ref{gts}, with topological and algebraic genericity, and spaceability of the set of such power series.}
\end{example}

The proof of Theorem \ref{gts} requires some preparation. Indeed, although such functions $f$ given by Theorem \ref{gts} will be referred to as generalized universal series because the approximation property is realized by series, the notion of generalized universal series under consideration in Theorem \ref{gts} is slightly different from that considered in the two previous sections. 

%

Let us introduce the formalism. We keep the notations of the beginning of Section \ref{S2}. We take $Y=\mathcal{H}\left(\mathbb{D}_R\right)$ for some $R\in \mathbb{R}_+ \cup\{+\infty\}$, $X=\mathcal{H}\left(\mathbb{C}\right)$ (as in Section 4, by Mergelyan's Theorem, the problem of approximating uniformly on some compact set $K\subset \mathbb{C}$ any function continuous on $K$, holomorphic in the interior of $K$, is equivalent to approximating any function in ${H}\left(\mathbb{C}\right)$), $Z={H}\left(\mathbb{C}\right)\times {H}\left(\mathbb{D}_R\right)$, and $M={H}\left(\mathbb{C}\right)\times \{0\}$. 
Consider a sequence $(\alpha_n)_{n\geq 1}$ of non-zero
complex numbers and define a fixed sequence $\mathcal{X}:=(x_{n,k})_{n\geq k\geq 0}$ in $H(\mathbb{C})$ as follows:
\begin{equation*}x_{2n+1,k}=0\hbox{ for }0\leq k\leq 2n+1\hbox{ and }x_{2n,k}=\left\{\begin{array}{ll}0&\hbox{ for }0\leq k< n\hbox{ and }k>2n,\\
\\
\displaystyle\frac{k!}{(k-n)!}\alpha_n z^{k-n}& \hbox{ for }n\leq k\leq 2n.\end{array}\right.
\end{equation*}
Clearly, for any $k\geq 0$, $x_{n,k}\rightarrow 0$ as $n$ tends to $+\infty$. As usual, we denote by $S_n^{\mathcal{X}}$ the map which takes a sequence of complex numbers $\left(a_k\right)_{k\geq 0}$ to $\sum_{k=0}^na_{k}x_{n,k}$, and we identify it with the map which carries $f=\sum_{k\geq 0}a_kz^k$ to $\sum_{k=0}^na_{k}x_{n,k}$. A straigtforward computation ensures that
\[
\alpha_{n} S_{n}\left(f^{(n)}\right)=S_{2n}^{\mathcal{X}}(f).
\]
With the notations of Section \ref{S2} in mind, we define the continuous linear mappings $L_n:Y\rightarrow Z,$ $f\mapsto (S_{2n}^{\mathcal{X}}(f),\hbox{id}-S_n(f)).$ Assumption (I) of Section \ref{S2} is satisfied because for a polynomial $p(z)=\sum_{i=0}^da_iz^i,$ we have 
$L_n(p)=(\alpha_nS_n(p^{(n)}),p-S_n(p))\rightarrow (0,0),$ as $p\rightarrow +\infty.$ We are now ready for the proof of Theorem \ref{gts}.

\begin{proof}[Proof of Theorem \ref{gts}] Let us first consider the case $R>0$. 
According to Theorem \ref{thms-27-and-28} combined with the continuity of the 
maps $L_n,$ the density of polynomials and Mergelyan's Theorem, the proof of Theorem \ref{gts}, except the assertion concerning the spaceability, will follow if we check that
for every $\varepsilon>0,$ 
every compact sets $K\subset\mathbb{C},$ with $K^c$ connected, every $L\subset\mathbb{D}_r,$ with $r<R,$ and every polynomial $h$, there exist $n\in\mu,$ $m\in\mathbb{N}$ with $m\geq n$ and
$a_0,a_1,\ldots,a_n,\ldots,a_{m}\in\mathbb{C}$ such that
\[
\sup_{z\in K}\left\vert\sum_{k=n}^{\min(m,2n)}a_k \frac{k!}{(k-n)!}\alpha_nz^{k-n}-h(z)\right\vert<\varepsilon,\ 
\sup_{z\in L}\left\vert\sum_{k=0}^{n}a_k z^{k}\right\vert<\varepsilon
\hbox{ and }
\sup_{z\in L}\left\vert\sum_{k=0}^{m}a_k z^{k}\right\vert<\varepsilon.
\]
Set $h(z)=\sum_{i=0}^db_iz^i$ and define $q_n(z)=\sum_{i=0}^db_i\frac{i!}{\alpha_n (i+n)!}z^{i+n}.$
Clearly one has
\[
\alpha_nS_n(q_n^{(n)})(z)=q(z)
\] 
for $n$ large enough and 
\[
\sup_{z\in L}\vert q_n(z)\vert \leq (d+1)(1+r^d)\max_{0\leq i\leq d}\vert b_i\vert\left(\frac{r^n}{\vert \alpha _n\vert n!}\right)
\rightarrow 0,\hbox{ as }n\rightarrow +\infty,
\] 
and 
\[
\sup_{z\in L}\vert S_n(q_n)(z)\vert \leq \vert b_0\vert\left(\frac{r^n}{\vert \alpha _n\vert n!}\right)
\rightarrow 0,\hbox{ as }n\rightarrow +\infty.
\]
Thus, let us choose $n\in \mu$ large enough, with $n\geq d,$ and $a_0=a_1=\ldots=a_{n-1}=0,$ 
$a_k=b_{k-n}\frac{(k-n)!}{k!\alpha_n}$ for $k=n,\ldots,d+n=m$ to obtain
the conclusion.

For the spaceability, we just need to apply Proposition \ref{closedgene} and to proceed as in the second paragraph of the proof of Theorem \ref{thm-gen-approx-lem}. Now for the case $R=0,$ the proof works along the same lines as above, except for the spaceability, which will proceed from Proposition \ref{propospacesequence1}. 
Indeed we can check that the sets $T_l$ (we refer the reader to Proposition \ref{propospacesequence1} for the 
definition) are dense in $H(\mathbb{C})$ by combining the density of polynomials with the above construction and 
the fact that $x_{2n,k}=0$ for $0\leq k<n.$ 
\end{proof}

Now we can wonder whether the radius of convergence of previous universal series is connected to the parameters $R.$ Indeed the universal approximation property is 
valid in $\mathbb{C},$ i.e. even inside the domain of holomorphy 
of the universal function when $R$ is strictly positive. 
In the following, we prove that the radius of convergence of generalized universal series in $\mathcal{W}_R^{(\mu)}(\alpha_n)$
is exactly $R.$ First of all we need of a combinatoric lemma.

\begin{lemma}\label{estima2} For every $0<\delta<1$ the following holds:
\[
\sum_{l=n}^{2n}\delta^{l-n}l(l-1)\dots(l-n+1)=n!\sum_{k=0}^{n}{2n+1\choose k}\delta^k(1-\delta)^{n-k}.
\]
\end{lemma}

\begin{proof}
Using Taylor's formula we get
\[
\sum_{l=2n+1}^{+\infty}\delta^{l-n}l(l-1)\dots(l-n+1)=\int_0^\delta
\frac{(2n+1)!}{(1-t)^{2n+2}}\frac{(\delta-t)^n}{n!}dt.
\] 
We put
$t=\delta-(1-\delta)x.$ We obtain
\[
\sum_{l=2n+1}^{+\infty}\delta^{l-n}l(l-1)\dots(l-n+1)=
\frac{(2n+1)!}{(1-\delta)^{n+1}n!}\int_0^{\delta/(1-\delta)}
\frac{x^n}{(1+x)^{2n+2}}dx.
\] 
Since we have the equality
\[
\frac{x^n}{(1+x)^{2n+2}}=-\frac{d}{dx}\left(
\frac{\sum_{k=0}^n\frac{n!^2}{k!(2n+1-k)!}x^k}
{(1+x)^{2n+1}}\right),
\] 
we deduce that the following equality holds
\begin{equation}\label{combinatoric}
\sum_{l=2n+1}^{+\infty}\delta^{l-n}l(l-1)\dots(l-n+1)=
\frac{n!}{(1-\delta)^{n+1}}-n!\sum_{k=0}^{n}{2n+1\choose
k}\delta^k(1-\delta)^{n-k}.
\end{equation}
Now observe that 
\[
\sum_{l=n}^{+\infty}\delta^{l-n}l(l-1)\dots(l-n+1)=\frac{n!}{(1-\delta)^{n+1}}
\]
and $\sum_{l=n}^{2n}=\sum_{l=n}^{+\infty}-\sum_{l=2n+1}^{+\infty}.$
By Equality (\ref{combinatoric}), the conclusion follows.
\end{proof}

\begin{proposition} Let $(\alpha_n)_{n\geq 1}$ be a sequence of non-zero complex numbers and 
$R\in\mathbb{R}_+\cup\{+\infty\}$ be the radius of
convergence of the power series $\sum_{n\geq 0}\frac{z^n}{n!\vert \alpha_n\vert}.$ 
Assume that, for every $r>R,$ we have $r^n/(\vert\alpha_n\vert n!)\rightarrow +\infty,$ as  
$n\rightarrow +\infty.$ Let $f\in \mathcal{W}_R^{(\mu)}(\alpha_n).$
Then the radius of convergence of its Taylor development at $0$ is equal to $R.$
\end{proposition}

\begin{proof} Let us write $f(z)=\sum_{n=0}^{+\infty}a_nz^n\in\mathcal{W}_R(\alpha_n)$. Let $R_0$ be its radius of convergence. According to Theorem \ref{gts} one has $R_0\geq R$. Without loss of generality, we may suppose that $R<+\infty.$ Assume that $R_0>R$. We have
\begin{equation}\label{roottest}\frac{1}{R_0}=\limsup\vert a_n\vert^{1/n}.
\end{equation}
Choose $\varepsilon>0$ such that $R(1+\varepsilon)<R_0.$ 
Equality (\ref{roottest}) implies that there exists $n_0\in\mathbb{N}$ such that for every $n\geq n_0$,
\[
\vert a_n\vert^{1/n}\leq
\frac{1}{R_0}(1+\varepsilon).
\]
Define $z_0\in\mathbb{C}\setminus\{0\}$ such that $\delta:=\frac{\vert z_0\vert}{R_0}(1+\varepsilon)<1$ and $\vert z_0\vert<\frac{R_0}{1+\varepsilon}-R.$
For all $n>n_0,$ the following estimate holds
\[
\begin{array}{rcl}\displaystyle\left\vert \alpha_n S_n\left(f^{(n)}\right)(z_0)\right\vert
&=& \displaystyle\left\vert \alpha_n \sum_{l=n}^{2n}a_l
l(l-1)\dots(l-n+1)z_0^{l-n}\right\vert\\&\leq&\displaystyle
\vert\alpha_n\vert \sum_{l=n}^{2n}\left(\vert a_l\vert^{1/l}\vert z_0\vert \right)^l
l(l-1)\dots(l-n+1)\vert z_0\vert ^{-n}\\&\leq&\displaystyle
\vert\alpha_n\vert \sum_{l=n}^{2n}\delta^l l(l-1)\dots(l-n+1)\vert z_0\vert
^{-n} \end{array}
\] 
Applying Lemma \ref{estima2}, we get
\[
\begin{array}{rcl}\displaystyle\left\vert \alpha_n S_n\left(f^{(n)}\right)(z_0)\right\vert
&\leq&\vert \alpha_n\vert n!\displaystyle\frac{\delta^n}{\vert
z_0\vert^n}\sum_{k=0}^n{2n+1 \choose
k}\delta^k(1-\delta)^{n-k}
\\&\leq&
\vert \alpha_n\vert n!\displaystyle\frac{\delta^n}{\vert
z_0\vert^n}\sum_{k=0}^{2n+1}{2n+1 \choose
k}\delta^k(1-\delta)^{n-k}\\&&
\\&\leq&
\displaystyle\frac{\vert\alpha_n\vert n!}{1-\delta}\displaystyle\left(\frac{\delta}{\vert
z_0\vert(1-\delta)}\right)^n.\end{array}
\]
Now $\vert z_0\vert<\frac{R_0}{1+\varepsilon}-R$ so $\frac{\delta}{\vert
z_0\vert(1-\delta)}<\frac{1}{R}$. Therefore the previous inequality shows that 
the sequence $(\alpha_n S_n(f^{(n)})(z_0))_{n\geq 0}$ is bounded. It is in contradiction with the universality of $f$, and $R_0=R.$
\end{proof}


\end{document}